\documentclass[12pt]{amsart}

\usepackage{amssymb}
\usepackage{amsthm}
\usepackage{mathrsfs}
\usepackage{amssymb}
\usepackage{amsfonts}
\usepackage[english]{babel}
\usepackage[T1]{fontenc}
\usepackage[latin1]{inputenc}
\usepackage{fullpage}
\usepackage{amsmath}
\usepackage[all]{xy}
\usepackage{stmaryrd}
\usepackage{color,xcolor}

\newtheorem{theorem}{Theorem}[section]

\theoremstyle{lemma}
\newtheorem{lemma}[theorem]{Lemma}

\theoremstyle{corollary}
\newtheorem{corollary}[theorem]{Corollary}

\theoremstyle{definition}
\newtheorem{definition}[theorem]{Definition}

\theoremstyle{proposition}
\newtheorem{proposition}[theorem]{Proposition}

\theoremstyle{conjecture}

\theoremstyle{remark}
\newtheorem{remark}[theorem]{Remark}

\numberwithin{equation}{section}

\def\sideremark#1{\ifvmode\leavevmode\fi\vadjust{\vbox to0pt{\vss
			\hbox to 0pt{\hskip\hsize\hskip1em
				\vbox{\hsize1.5cm\tiny\raggedright\pretolerance10000
					\noindent #1\hfill}\hss}\vbox to8pt{\vfil}\vss}}}

\begin{document}

\title{Volume comparison with respect to scalar curvature}

\thanks{This work was supported by NSFC (Grant No. 12071489, No. 11601531, No. 11521101).}


\author{Wei Yuan}
\address{(Wei Yuan) Department of Mathematics, Sun Yat-sen University, Guangzhou, Guangdong 510275, China}
\email{yuanw9@mail.sysu.edu.cn}



\dedicatory{Dedicated to Nankai Univerisity for her 100th - Anniversary.\\
	(1919 - 2019)}


\begin{abstract}
In this article, we investigate the volume comparison with respect to scalar curvature. In particular, we show volume comparison holds for small geodesic balls of metrics near a $V$-static metric. For closed manifold, we prove the volume comparison for metrics near a strictly stable Einstein metric. As applications, we give a partial answer to a conjecture of Bray and recover a result of Besson, Courtois and Gallot, which partially confirms a conjecture of Schoen about closed hyperbolic manifold. Applying analogous techniques, we obtain a different proof of a local rigidity result due to Dai, Wang and Wei, which shows it admits no metric with positive scalar curvature near strictly stable Ricci-flat metrics.
\end{abstract}

\maketitle



\section{Introduction}

The volume comparison theorem is a fundamental result in Riemannian geometry. It is a powerful tool in geometric analysis and frequently used in solving various problems. \\

The classic volume comparison theorem states that the volume of a complete manifold is upper bounded by the round sphere if its Ricci curvature is lower bounded by a corresponding positive constant. A natural question is that whether we can replace the assumption on Ricci curvature by the one with scalar curvature? \\

In general, scalar curvature is not sufficient to control the volume. This is a straightforward conclusion of a result by Corvino, Eichmair and Miao \cite{C-E-M}. In order to state it, we need the following fundamental concept, which was introduced by Miao and Tam in \cite{M-T_1}:

\newtheorem*{V-static}{\bf Definition}
\begin{V-static}
Let $(M^n, \bar g)$ be an $n$-dimensional Riemannian manifold. We say $\bar g$ is a $V$-static metric if there is a smooth function $f \not\equiv 0$ and a constant $\kappa \in \mathbb{R}$ solve the following \emph{$V$-static equation}:
\begin{align}\label{eqn:V_static}
\gamma_{\bar g}^* f = \nabla^2_{\bar g} f - \bar g \Delta_{\bar g} f - f Ric_{\bar g} =\kappa \bar g,
\end{align}
where $\gamma_{\bar g}^* : C^\infty(M) \rightarrow S_2(M)$ is the formal $L^2$-adjoint of $\gamma_{\bar g}:=DR_{\bar g}$, the linearization of scalar curvature at $\bar g$. We also refer a quadruple $(M, \bar g, f, \kappa)$ to be a \emph{$V$-static space}.
\end{V-static}

\begin{remark}\label{rmk:scalar_curv_V_static}
A fundamental property of $V$-static metric is that its scalar curvature $R_{\bar g}$ is a constant for $M$ connected (see Proposition 2.1 in \cite{C-E-M}). By taking trace of the equation (\ref{eqn:V_static}), we can see that $f$ satisfies the linear elliptic equation
\begin{align}\label{eqn:V_static_trace}
\Delta_{\bar g} f + \frac{R_{\bar g}}{n-1} f + \frac{n\kappa }{n-1} =0.
\end{align}
In particular, $f$ is an eigenfunction for the Laplacian, if $\kappa = 0$. 
\end{remark}

Einstein metrics are in particular $V$-static, which can be easily seen by taking the function $f$ to be a constant. In this sense, we can view $V$-static metrics as a generalization of Einstein metrics. Another class of special $V$-static metrics are \emph{vacuum static metrics} when we take $\kappa = 0$. They can be used to construct an important category of solutions to \emph{Einstein field equation} in general relativity \cite{Q-Y_1}. The classification of $V$-static spaces is a crucial problem in understanding the interplay between scalar curvature and volume. For more results, please refer to \cite{B-R, B-D-R, C-E-M, M-T_1, M-T_2}.\\

Now we state a deformation result associated to the concept of $V$-static metric:
\begin{theorem}[Corvino, Eichmair and Miao \cite{C-E-M}]
Let $(M^n, \bar g)$ be a Riemannian manifold and $\Omega \subset M$ be a pre-compact domain with smooth boundary. Suppose $(\Omega, \bar g)$ is not $V$-static, i.e the $V$-static equation (\ref{eqn:V_static}) only admits the trivial solution : $f \equiv 0$ and $\kappa = 0$ in $C^\infty(\Omega) \times \mathbb{R}$. Then for any $\Omega_0$ compactly contained in $\Omega$, there exists a constant $\delta_0 > 0$ such that for any $(\rho, V) \in C^\infty(M) \times \mathbb{R}$ with $supp\ (R_{\bar g} - \rho) \subset \Omega_0$ and
$$||R_{\bar g} - \rho||_{C^1(\Omega, \bar g)}+ |V_\Omega(\bar g) - V| < \delta_0,$$
there exists a metric $g$ on $M$ such that $supp\ (g - \bar g) \subset \Omega$, $R_g = \rho$ and $V_\Omega(g) = V$.
\end{theorem}

This result suggests that for a non-$V$-static domain, the information of scalar curvature is not sufficient to give a volume comparison: we can take either $V > V_\Omega(\bar g)$ or $V < V_\Omega(\bar g)$, but with $\rho > R_{\bar g}$ in $\Omega$. In either case, we can find a metric $g$ realizing $(\rho, V)$ on $\Omega$ and it shows that no volume comparison holds for non-$V$-static domains.\\

However, the volume comparison with respect to scalar curvature indeed holds for some special metrics. For instance, Miao and Tam proved a rigidity result for upper hemisphere with respect to non-decreasing scalar curvature and volume. They also showed that a similar result holds for Euclidean domains \cite{M-T_2}. Note that all space forms are $V$-static, it is natural to ask whether all $V$-static spaces admit such a volume comparison result. \\

Inspired by the rigidity of vacuum static metrics \cite{Q-Y_2} and related work \cite{M-T_2}, we obtain a volume comparison theorem with respect to scalar curvature for sufficiently small geodesic balls, if appropriate boundary conditions on induced metric $g|_{T\partial B_r(p)}$ and mean curvature $H_g$ are posed. 
\newtheorem*{thm_A}{\bf Theorem A}
\begin{thm_A}\label{thm:loc_vol_comparison}
For $n\geq 3$, suppose $(M^n, \bar g, f, \kappa)$ is a $V$-static space. For any $p \in M$ with $f(p) > 0$, there exist positive constants $r_0$ and $\varepsilon_0$ such that for any geodesic ball $B_r(p) \subset M$ with radius $r \in (0, r_0)$ and metric $g$ on $B_r(p)$ satisfying
\begin{itemize}
\item $R_g \geq R_{\bar g}$ in $B_r(p)$,
\item $H_g \geq H_{\bar g}$ on $\partial B_r(p)$,
\item $g|_{T\partial B_r(p)} = \bar g|_{T\partial B_r(p)}$,
\item $||g - \bar g||_{C^2(B_r(p), \bar g)} < \varepsilon_0$,
\end{itemize} 
the following volume comparison holds:
\begin{itemize}
\item if $\kappa < 0$, then $$V_\Omega(g) \leq V_\Omega (\bar g);$$
\item if $\kappa > 0$, then $$V_\Omega(g) \geq V_\Omega (\bar g);$$
\end{itemize} 
with equality holding in either case if and only if the metric $g$ is isometric to $\bar g$.
\end{thm_A}

\begin{remark}
If $f(p) < 0$, we only need to replace $(f, \kappa)$ by $(-f, -\kappa)$, then the reversed volume comparison follows.
\end{remark}

\begin{remark}
If $\kappa = 0$, $V$-static metrics are in particular vacuum static and hence $g$ is isometric to $\bar g$ according to \cite{Q-Y_2}. Thus Theorem A is an extension for the rigidity of vacuum static metrics.
\end{remark}
\vskip 0.2in

In general, the function $f$ may change its sign on a closed $V$-static manifold. For example, we can take $f:= 1 + 2 x_{n+1}$ on the unit sphere $\mathbb{S}^n$, where $x_{n+1}$ is the height-function of $\mathbb{S}^n \hookrightarrow \mathbb{R}^{n+1}$. Hence the volume comparison may not hold in this case. However, for some special $V$-static spaces, the volume comparison with respect to scalar curvature might still hold for closed manifolds. Here and throughout this article, we refer a manifold to be \emph{closed}, if it is compact without boundary.\\

In \cite{Schoen}, Schoen proposed a well-known conjecture that the Yamabe invariant of a given closed hyperbolic manifold is achieved by its canonical metric. This problem involves all possible metrics on a given hyperbolic manifold and it is obviously very difficult to solve. However, it can be shown that this conjecture is in fact equivalent to the following volume comparison problem:
\newtheorem*{conj}{\bf Schoen's Conjecture}
\begin{conj}
	For $n \geq 3$, let $(M^n, \bar g)$ be a closed hyperbolic manifold. Then for any metric $g$ on $M$ with $$R_g \geq R_{\bar g},$$ the volume comparison $$V_M (g) \geq V_M(\bar g)$$
	holds.
\end{conj}

The equivalence of aforementioned Schoen's conjectures are known by experts. For the convenience of readers, we include a proof in the appendix.\\

Schoen's conjecture is known holding for $3$-manifolds due to works of Hamilton on non-singular Ricci flow \cite{Hamilton} and Perelman on geometrization of 3-manifolds \cite{Perelman_1, Perelman_2} (also see \cite{A-S-T} for a generalization). For higher dimensions, Besson, Courtois and Gallot verified it for metrics $C^2$-close to the canonical metric \cite{B-C-G_1}. They also proved that the volume comparison holds without assuming $g$ is close to $\bar g$, if one replace the assumption on scalar curvature by Ricci curvature \cite{B-C-G_2}, which can be viewed as an evidence that Schoen's conjecture holds for higher dimensions. \\

For the case of positive curvature, Bray proposed the following conjecture:
\newtheorem*{conj_B}{\bf Bray's Conjecture}
\begin{conj_B}
	For $n \geq 3$, there is a positive constant $\varepsilon_n < 1$ such that for any complete manifold $(M^n, g)$ with scalar curvature $$R_g \geq n(n-1)$$ and Ricci curvature $$Ric_g \geq \varepsilon_n (n-1) g,$$ the volume comparison $$V_M (g) \leq V_{\mathbb{S}^n}(g_{_{\mathbb{S}^n}})$$
	holds, where $\mathbb{S}^n$ is the unit round sphere and $g_{_{\mathbb{S}^n}}$ is the canonical metric.
\end{conj_B}

\begin{remark}
	Unlike Schoen's conjecture, there is an additional assumption on Ricci curvature in the positive curvature case. In fact, this assumption is necessary, see \cite{Bray} for more details.
\end{remark}
\vskip 0.2in

	 For this conjecture, Bray \cite{Bray} verified it for three dimensional manifolds and give an estimate for $\varepsilon_3$. Later, Gursky and Viaclovsky \cite{Gursky_Viaclovsky} showed that $\varepsilon_3 \leq \frac{1}{2}$  and Brendle \cite{Brendle} proved the rigidity of volume comparison for $\varepsilon_3 = \frac{1}{2}$. For higher dimensions, Zhang gave a partial answer in \cite{Zhang}.\\

Before stating our result, we first recall the following well-known concept associated to an Einstein metric:

\begin{definition}[Stability of Einstein metric]
	For $n\geq 3$, suppose $(M^n, \bar g)$ is a closed Einstein manifold. The metric $\bar g$ is said to be \emph{stable}, if 
	\begin{align}\label{def:Einstein_stability}
		\min \text{spec}_{_{TT}} ( - \Delta_E^{\bar g}) = \underset{0\not\equiv h \in S_2^{_{TT}}(M)}{\inf} \frac{\int_M \langle h, -\Delta_E^{\bar g} h \rangle_{\bar g} dv_{\bar g}}{\int_M |h|^2_{\bar g} dv_{\bar g}} \geq 0,
	\end{align}	
	where $\Delta_E^{\bar g} : = \Delta_{\bar g} + 2Rm_{\bar g}$ is the \emph{Einstein operator} and
	$$S_{2,\bar g}^{_{TT}}(M):= \{ h \in S_2(M) : \delta_{\bar g} h = 0,\ tr_{\bar g} h = 0\}$$
	is the space of \emph{transverse-traceless symmetric $2$-tensors} on $(M,\bar g)$. Moreover, $\bar g$ is called \emph{strictly stable} if the inequality in (\ref{def:Einstein_stability}) is strict.
\end{definition}

Stability is a crucial concept in the study of Einstein manifolds. There are several equivalent way to define it, we adopt the one involving the Einstein operator for our convenience. For more information, please refer to \cite{Besse, D-W-W_1, D-W-W_2, Kroncke}.\\

Our main result about volume comparison for Einstein manifolds is the following:
\newtheorem*{thm_B}{\bf Theorem B}
\begin{thm_B}\label{thm:Einstein_vol_comparison}
Suppose $(M^n, \bar g)$ is a closed strictly stable Einstein manifold with $$Ric_{\bar g} = (n-1)\lambda \bar g,$$ where $\lambda \neq 0$ is a constant. There exists a constant  $\varepsilon_0 > 0$ such that for any metric $g$ on $M$ which satisfies 
$$R_g \geq n(n-1)\lambda$$
and
$$||g - \bar g||_{C^2(M, \bar g)} < \varepsilon_0,$$
the following volume comparison holds:
\begin{itemize}
	\item if $\lambda > 0$, then $$V_M (g) \leq V_M (\bar g);$$
	\item if $\lambda < 0$, then $$V_M (g) \geq V_M (\bar g).$$
\end{itemize} 
Moreover, the equality holds in either case if and only if the metric $g$ is isometric to $\bar g$.
\end{thm_B}

\begin{remark}
	Suppose the reference metric $\bar g$ is K\"ahler-Einstein with negative scalar curvature and all infinitesimal complex deformations of its complex structure are integrable. Applying a delicate utilization of the functional
	\begin{align*}
		K (g) = \int_M |R_g|^{\frac{n}{2}} dv_g
	\end{align*}
	and the Yamabe functional
	\begin{align*}
		Y(g) = \frac{\int_M R_g dv_g}{(V_M (g))^{\frac{n-2}{n}}},
	\end{align*}
	Dai, Wang and Wei proved that the volume comparison with respect to scalar curvature holds for metrics near $\bar g$ (see Theorem 1.5 in \cite{D-W-W_2}). In fact, their result can be extended to strictly stable Einstein metrics with negative scalar curvature.
\end{remark}

\begin{remark}
	The above volume comparison does not hold for Ricci-flat metrics: by taking $g = c^2 \bar g$ for a constant $c > 0$, we have the scalar curvature $R_g = R_{\bar g} = 0$, but the volume $V_M(g)$ can be either larger or smaller than $V_M(\bar g)$ depending on either $c > 1$ or $c < 1$.
\end{remark}

\begin{remark}
	The stability assumption in the theorem is necessary. Macbeth constructed an example of Einstein manifold which shows the volume comparison fails if we lack stability \cite{Macbeth}. See Proposition \ref{prop:heather_example} for more details.
\end{remark}

\begin{remark}
	Our approach in fact works for other curvatures as well. Please see \cite{Lin-Yuan} for a volume comparison theorem of $Q$-curvature for strictly stable positive Einstein manifolds.
\end{remark}
\vskip 0.2in

It is well-known that hyperbolic metrics are strictly stable as special Einstein metrics and hence Theorem B provides a partial answer to Schoen's conjectures, which recovers the result in \cite{B-C-G_1}:
\newtheorem*{cor_A}{\bf Corollary A}
\begin{cor_A}[Besson, Courtois and Gallot \cite{B-C-G_1}]\label{cor:loc_vol_comparison_hyperbolic_mfd}
For $n \geq 3$, let $(M^n, \bar g)$ be a closed hyperbolic manifold. There exists a constant $\varepsilon_0 > 0$ such that for any metric $g$ on $M$ with scalar curvature $$R_g \geq R_{\bar g}$$ and $$||g - \bar g||_{C^2(M, \bar g)} < \varepsilon_0,$$ we have $$V_M (g) \geq V_M(\bar g)$$
with equality holds if and only if the metric $g$ is isometric to $\bar g$.\\
\end{cor_A}

Similarly, the spherical metric is also strictly stable (Example 3.1.2 in \cite{Kroncke}), we obtain a partial answer to Bray's conjecture:
\newtheorem*{cor_B}{\bf Corollary B}
\begin{cor_B}\label{cor:loc_vol_comparison_sphere}
For $n \geq 3$, let $(\mathbb{S}^n, g_{_{\mathbb{S}^n}})$ be the unit round sphere. There exists a constant $\varepsilon_0 > 0$ such that for any metric $g$ on $\mathbb{S}^n$ with scalar curvature $$R_g \geq n(n-1)$$ and $$||g - g_{_{\mathbb{S}^n}}||_{C^2(\mathbb{S}^n, g_{_{\mathbb{S}^n}})} < \varepsilon_0,$$ we have $$V_{\mathbb{S}^n} (g) \leq V_{\mathbb{S}^n}(g_{_{\mathbb{S}^n}})$$
with equality holds if and only if the metric $g$ is isometric to $g_{_{\mathbb{S}^n}}$.
\end{cor_B}

\begin{remark}
	For metrics close to the canonical spherical metric, the assumption on Ricci curvature in Bray's conjecture holds automatically.
\end{remark}

\begin{remark}
	Corvino, Eichmair and Miao constructed a metric on the upper hemisphere which satisfies the scalar comparison but has arbitrarily large volume (see Proposition 6.2 in \cite{C-E-M}). In fact, by gluing a lower hemisphere, we can get a metric on the whole sphere with scalar curvature no less than $n(n-1)$ but has larger volume. 
\end{remark}
\vskip 0.2in

In the research of scalar curvature, a fundamental question is that whether a given manifold admits a metric of positive scalar curvature. A well-known result due to Schoen and Yau \cite{S-Y_1, S-Y_2}, Gromov and Lawson \cite{G-L_1, G-L_2} is the rigidity of tori, which states that there is no metric of positive scalar curvature on tori. For an excellent survey, please refer to \cite{Rosenberg}. \\

In \cite{D-W-W_1}, Dai, Wang and Wei studied the existence of metrics with positive scalar curvature on a Riemannian manifold with nonzero parallel spinors. Through investigations of variational properties for the first eigenvalue of conformal Laplacian, they proved the local rigidity of scalar curvature near the reference metric. Note that their proof is in fact can be applied to closed strictly stable Ricci-flat manifolds.\\

Applying similar techniques of the argument for Theorem B, we obtain the local rigidity of strictly stable Ricci-flat manifolds, which generalizes a result of Fischer and Marsden about local rigidity of tori\cite{F-M} with a different approach than \cite{D-W-W_1}:
\newtheorem*{thm_C}{\bf Theorem C}
\begin{thm_C}[Dai, Wang and Wei \cite{D-W-W_1}]\label{thm:Ricci_flat_rigidity}
	Suppose $(M^n, \bar g)$ is a strictly stable Ricci-flat manifold, then there exists a constant $\varepsilon_0 > 0$ such that for any metric $g$ on $M$ satisfies $$R_g \geq 0$$
	and 
	$$||g - \bar g||_{C^2(M, \bar g)} < \varepsilon_0,$$
	then $g$ is homothetic to $\bar g$. That is, we can find a constant $c > 0$ such that $g = c^2\bar g$. In particular, there is no metric with positive scalar curvature near $\bar g$. 
\end{thm_C}

\begin{remark}
	Note that flat tori are merely stable, since the kernel of Einstein operator is non-trivial and in fact $$\dim \ \ker \Delta_E^{\bar g} \geq \frac{n(n+1)}{2} - 1.$$ It is interesting to see whether there is an example of closed stable Ricci-flat manifold which admits a metric of positive scalar curvature near the reference metric. 
\end{remark}

\begin{remark}
	Similar to Theorem B, our approach can also be applied to other curvatures. Please see \cite{Lin-Yuan} for an analogous result for $Q$-curvature.
\end{remark}
\vskip 0.2in

The article is organized as follow: In Section 2, we collect notations and conventions used frequently in this article. In Section 3, we calculate some geometric variational formulas involved in next two sections. In Section 4, we study the volume comparison for geodesic balls in $V$-static spaces. In Section 5, we investigate the volume comparison for non-Ricci-flat strictly stable Einstein manifolds and the rigidity phenomenon of strictly stable Ricci-flat manifolds. In Appendix A, we present a proof for equivalence of two conjectures proposed by Schoen.\\

\paragraph{\textbf{Acknowledgement}}

The author would like to express his appreciations to Professor Miao Pengzi for suggesting this interesting problem. The author also would like to thank Professors Chen Bing-Long, Lin Yueh-Ju, Heather Macbeth, Qing Jie, Wei Guofang, Zhang Hui-Chun and Zhu Xi-Ping for their interests in this problem and inspiring discussions. In particular, the author would like to express his appreciations to Professor Zhang Hui-Chun for pointing out the work \cite{B-C-G_1}, Professor Zhu Xi-Ping for explaining the idea of the work \cite{A-S-T} and also many other valuable comments, Professor Heather Macbeth for showing the counter example for unstable Einstein manifolds and suggestions for considering the stability of Einstein metrics, Professor Wei Guofang for valuable comments. The author also would like to thank the referee for many valuable comments and suggestions on this article.\\


\section{Notations and conventions}

In this section, we collect notations frequently used and conventions adopted in this article for the convenience of readers. Please note that {\bfseries{all calculations are performed in the reference metric $\bar g$.}}\\

Let $(\Omega^n, \bar g)$ be an $n$-dimensional compact manifold possibly with $C^1$-boundary $\Sigma:=\partial \Omega$:
\begin{enumerate}
\item Indices of coordinates components:
\begin{itemize}
	\item Greek indices run through $1, \cdots, n$;
	
	\item Latin indices run through $1, \cdots, n-1$.
\end{itemize}

\item Connections:
\begin{itemize}
	\item connection on $\Omega$ with respect to $\bar g$:\quad $\nabla_{\bar g}$;
	
	\item connection on $\Sigma$ with respect to $\bar g|_{_{T\Sigma}}$:\quad $\nabla^{\Sigma}$.
\end{itemize}

\item Volume forms:
\begin{itemize}
	\item volume form on $\Omega$ with respect to $\bar g$:\quad $dv_{\bar g}$;
	
	\item volume form on $\Sigma$ with respect to $\bar g|_{_{T\Sigma}}$:\quad $d\sigma_{\bar g}$.
\end{itemize}

\item Curvatures:
\begin{itemize}
	\item Riemann curvature tensor $Rm_{\bar g}$:\quad  $R_{\alpha \beta \gamma \delta}$;
	
	\item Ricci curvature tensor $Ric_{\bar g}$:\quad  $R_{\beta \gamma} = \bar g^{\alpha \delta} R_{\alpha \beta \gamma \delta}$;
	
	\item scalar curvature $R_{\bar g}$:\quad  $R_{\bar g} = \bar g^{\beta \gamma} R_{\beta \gamma}$;
	
	\item second fundamental form $A_{\bar g}$:\quad  $A_{ij}^{\bar g} = \frac{1}{2} \partial_{{\nu_{\bar g}}}\bar g_{ij}$;
	
	\item mean curvature $H_{\bar g}$:\quad  $H_{\bar g} = \bar g^{ij} A_{ij}^{\bar g}$. 
	
\end{itemize}

\item Spaces:
\begin{itemize}
	\item space of all smooth Riemannian metrics on $\Omega$:\quad $\mathcal{M}_\Omega$;
	
	\item space of all smooth diffeomorphisms of $\Omega$:\quad $\mathscr{D}(\Omega)$;
	
	\item a local slice through the metric $\bar g$:\quad $\mathcal{S}_{\bar g}$;
	
	\item symmetric $2$-tensors on $\Omega$:\quad  $S_2(\Omega)$;
	
	\item $TT$-tensors on $(\Omega, \bar g)$:\quad  $S_{2,\bar g}^{_{TT}}(\Omega) = \{ h \in S_2(\Omega): \ \delta_{\bar g} h = 0, \ tr_{\bar g} h = 0 \}$.
\end{itemize}

\item Operators:
\begin{itemize}
	\item Multiplication and inner product of symmetric $2$-tensors:
	\begin{align*}
		(h \times k)_{\alpha \delta}:= \bar g^{\beta \gamma} h_{\alpha \beta} k_{\gamma \delta}
	\end{align*}
	and 
	\begin{align*}
		 \langle h, k \rangle_{\bar g} = h \cdot k := \bar g^{\alpha \delta} (h \times k)_{\alpha \delta} = \bar g^{\alpha \delta} \bar g^{\beta \gamma} h_{\alpha \beta} k_{\gamma \delta} .
	\end{align*}
	In particular,
	\begin{align*}
		(h^2)_{\alpha\beta} = \bar{g}^{\gamma\delta}h_{\alpha \gamma} h_{\delta\beta} 
	\end{align*}
	and
	\begin{align*}
		Ric_{\bar g} \cdot h : = R_{\beta \gamma} h^{\beta \gamma}.
	\end{align*}
	\item Riemann curvature tensor as operator on symmetric $2$-tensors: 
	\begin{align*}
		(Rm_{\bar g} \cdot h)_{\beta\gamma}:= R_{\alpha \beta \gamma \delta} h^{\alpha \delta}
	\end{align*}
	and
	\begin{align*}
		\langle Rm_{\bar g} \cdot h, h\rangle_{\bar g}: = R_{\alpha\beta\gamma\delta}h^{\alpha\delta}h^{\beta\gamma}.
	\end{align*}
	
	\item A combination involving curvature:
	$$\mathscr{R}_{\bar g}(h, h)  := \langle Rm_{\bar g} \cdot h, h \rangle_{\bar g}  + 2(Ric_{\bar g}\cdot h)(tr_{\bar g} h) -\frac{2 R_{\bar g}}{n-1}(tr_{\bar g} h)^2 .$$
	
	\item Formal $L^2$-adjoint of covariant differentiation:
	\begin{align*}
		\delta_{\bar g} = - div_{\bar g}: \quad (\delta_{\bar g} h)_\beta = - \nabla^\alpha_{\bar g} h_{\alpha\beta}.
	\end{align*}	
	\item Einstein operator: 
	\begin{align*}
		\Delta_E^{\bar g} h = \Delta_{\bar g} h + 2 Rm_{\bar g} \cdot h.
	\end{align*}	
	\item Linearization of scalar curvature: 
	\begin{align*}
		\gamma_{\bar g} h = - \Delta_{\bar{g}} (tr_{\bar{g}} h) + \delta_{\bar{g}}^2 h - Ric_{\bar{g}} \cdot h.
	\end{align*}	
	\item Formal $L^2$-adjoint of $\gamma_{\bar g}$ : 
	\begin{align*}
		\gamma_{\bar g}^* f = \nabla_{\bar g}^2 f - \bar g \Delta_{\bar g} f - f Ric_{\bar g}.
	\end{align*}
\end{itemize}
\end{enumerate}

\ \\


\section{Geometric variational formulas}

In this section, we give variational formulas for geometric functionals involved later in the argument. Throughout this section, $\Omega$ is assumed to be a compact manifold possibly with $C^1$-boundary $\Sigma:= \partial \Omega$. In case of $\Sigma \neq  \emptyset$, let 
$$\{ e_1, \cdots, e_{n-1}, e_n = {\nu_{\bar g}}\}$$ 
be an orthonormal frame on $\Sigma$ such that $\{ e_i \}_{i=1}^{n-1}$ are tangent to $\Sigma$ and ${\nu_{\bar g}}$ is the outward normal vector field of $\Sigma$ with respect to the metric $\bar g$. We also denote the induced connection on $\Sigma$ by $\nabla^\Sigma$. \\

We begin with recalling well-known variational formulas of scalar curvature (for detailed calculations, please refer to \cite{F-M, Yuan}):
\begin{lemma}\label{lem:scalar_variation_formulae}
	The first and second variations of scalar curvature are
	\begin{align}
	DR_{\bar{g}} \cdot h = - \Delta_{\bar{g}} (tr_{\bar{g}} h) + \delta_{\bar{g}}^2 h - Ric_{\bar{g}} \cdot h,
	\end{align}
	and 
	\begin{align}
	 D^2 R_{\bar{g}} \cdot (h,h) =& -2 \gamma_{\bar{g}} (h^2) - \Delta_{\bar g} |h|_{\bar{g}}^2 - \frac{1}{2} |\nabla_{\bar{g}}h|_{\bar{g}}^2 - \frac{1}{2}|d(tr_{\bar{g}}h)|_{\bar{g}}^2\\ \notag &+ 2 \langle h, \nabla^2_{\bar{g}} (tr_{\bar{g}} h) \rangle_{\bar{g}} - 2 \langle \delta_{\bar{g}} h, d(tr_{\bar{g}} h) \rangle_{\bar{g}} + \nabla_{\alpha} h_{\beta\gamma} \nabla^{\beta} h^{\alpha\gamma}
	\end{align}
	for any $h \in S_2(\Omega)$.\\
\end{lemma}

For mean curvature, its variations for fixed induced boundary metric are given as follow, which was first shown by Brendle and Marques in \cite{B-M}:
\begin{lemma}\label{lem:mean_curv_variations}
	The first and second variations of mean curvature are 
	\begin{align}
	DH_{\bar g} \cdot h =\ \frac{1}{2} h_{nn} H_{\bar{g}} - \nabla_i h_n^{\ i} + \frac{1}{2} \nabla_n h_i^{\ i}
	\end{align}
	and
	\begin{align}
	D^2H_{\bar{g}} \cdot (h, h) =  \left( -\frac{1}{4} h^2_{nn} + \sum_{i=1}^{n-1} h^2_{in} \right) H_{\bar{g}} + h_{nn}\left( \nabla_i h_n^{\ i} - \frac{1}{2}\nabla_n h_i^{\ i}\right)
	\end{align}
	for any $h \in S_2(\Omega)$ with $\left.h\right|_{T\partial \Omega} \equiv 0$.\\
\end{lemma}

For the volume functional, we provide a proof mainly based on a technique from linear algebra, which would be useful in calculating higher order variational formulas.
\begin{lemma}\label{lem:vol_variations}
	The first and second variations of volume are 
	\begin{align}
	DV_{\Omega, \bar g} \cdot h = \frac{1}{2} \int_\Omega (tr_{\bar g} h) dv_{\bar g}
	\end{align}
	and
	\begin{align}
	D^2V_{\Omega, \bar g} \cdot (h, h) = \frac{1}{4} \int_{\Omega} \left[ (tr_{\bar g} h)^2 - 2|h|_{\bar g}^2 \right]dv_{\bar g}
	\end{align}
	for any $h \in S_2(\Omega)$.
\end{lemma}

\begin{proof}
	Let $A$ be an $n \times n$ symmetric matrix. Its characteristic polynomial is given by
	\begin{align*}
	p_A(\lambda) &= \det (\lambda I - A ) \\
	&=  \sum_{k=0}^n (-1)^k \sigma_k(A) \lambda^{n-k}\\
	&= \lambda^n - (tr A) \lambda^{n-1} + \frac{1}{2} ((tr A)^2 - tr A^2) \lambda^{n-2} + \sum_{k=3}^n (-1)^k \sigma_k(A) \lambda^{n-k},
	\end{align*}
	where $\sigma_k(A)$ is the $k^{th}$-elementary symmetric polynomial associated to the matrix $A$.
	
	Choosing normal coordinates with respect to $\bar g$ centered at an interior point $x \in \Omega$, so that $\bar g_{\alpha\beta} = \delta_{\alpha\beta}$ at $x$. From the linear algebra fact mentioned above, we have the expansion
	$$\det (\bar g + h) = 1 + (tr_{\bar g} h) + \frac{1}{2} ( (tr_{\bar g} h)^2 - |h|_{\bar g}^2 ) + O(|h|_{\bar g}^3)$$ and hence $$\sqrt{\det ( \bar g + h )} =  1 + \frac{1}{2}(tr_{\bar g} h) + \frac{1}{8} ( (tr_{\bar g} h)^2 - 2 |h|_{\bar g}^2 )  + O(|h|_{\bar g}^3).$$
	Immediately, it implies
	\begin{align*}
	DV_{\Omega, \bar g} \cdot h = \frac{1}{2} \int_\Omega (tr_{\bar g} h) dv_{\bar g}
	\end{align*}
	and
	\begin{align*}
	D^2V_{\Omega, \bar g} \cdot ( h, h) = \frac{1}{4} \int_\Omega \left( (tr_{\bar g} h)^2 - 2 |h|_{\bar g}^2 \right) dv_{\bar g}
	\end{align*} 
	respectively.
\end{proof}

\vskip 0.2in

In the rest of this section, we calculate variational formulas for some particularly designed functionals involving scalar curvature, mean curvature and volume.

\begin{proposition}\label{prop:int_scalar_first variation}
	For any $h \in S_2(\Omega)$ and $f\in C^\infty(\Omega)$, 
	\begin{align*}
	& \int_\Omega \left(DR_{\bar g} \cdot h\right) f dv_{\bar{g}} \\
	=& \int_\Omega \langle h, \gamma^*_{\bar{g}} f \rangle_{\bar g}  dv_{\bar{g}} + \int_\Sigma \left[ - (\partial_{{\nu_{\bar g}}} (tr_{\bar g} h) + \langle \delta_{\bar g} h,  {\nu_{\bar g}} \rangle_{\bar g} ) f + (tr_{\bar g} h) \partial_{{\nu_{\bar g}}}f  - h( {\nu_{\bar g}}, \nabla_{\bar g} f )  \right]d\sigma_{\bar g}.
	\end{align*}
\end{proposition}

\begin{proof}
	It is straightforward that
	\begin{align*}
	& \int_\Omega \left(DR_{\bar g} \cdot h\right) f dv_{\bar{g}} \\
	=& \int_\Omega \left( - \Delta_{\bar g}(tr_{\bar g} h) + \delta_{\bar g}^2 h - Ric_{\bar g} \cdot h \right) f dv_{\bar g}\\
	=& \int_\Omega \langle h, \gamma^*_{\bar{g}} f \rangle_{\bar g}  dv_{\bar{g}} + \int_\Sigma \left[ - (\partial_{{\nu_{\bar g}}} (tr_{\bar g} h) + \langle \delta_{\bar g} h,  {\nu_{\bar g}} \rangle_{\bar g} ) f + (tr_{\bar g} h) \partial_{{\nu_{\bar g}}}f  - h( {\nu_{\bar g}}, \nabla_{\bar g} f )  \right]d\sigma_{\bar g}
	\end{align*}
	from Lemma \ref{lem:scalar_variation_formulae} and integration by parts.
\end{proof}

\vskip 0.2in

\begin{proposition}\label{prop:second_var_integral_scalar_curvature}
	For any $h \in S_2(\Omega)$ and $f\in C^\infty(\Omega)$, 
	\begin{align*}
	&\int_\Omega ( D^2 R_{\bar g} \cdot (h,h) )f dv_{\bar g}\\
	=&  \int_\Omega \left[ - \frac{1}{2} |\nabla_{\bar g} h|_{\bar g}^2 - \frac{1}{2}|d(tr_{\bar g} h)|_{\bar g}^2 + |\delta_{\bar g} h|^2 - 2 \langle \delta_{\bar g} h, d(tr_{\bar g} h) \rangle_{\bar g} + 2(tr_{\bar g} h) (\delta_{\bar g}^2 h)  + \mathscr{R}_{\bar g}(h, h)  \right] f  dv_{\bar g}\\
	&+\int_\Omega \left[ 2(tr_{\bar g} h) \left( \langle h, \gamma_{\bar g}^*f \rangle_{\bar g} - 2 \langle \delta_{\bar g} h, d f \rangle_{\bar g} - \frac{1}{n-1} (tr_{\bar g} h) \left(tr_{\bar g} (\gamma_{\bar g}^* f) \right) \right) - 2 \langle h, \delta_{\bar g} h \otimes df \rangle_{\bar g} - \langle\gamma^*_{\bar{g}} f, h^2\rangle_{\bar g} \right] dv_{\bar g} \\
	&+ \int_\Sigma \left [  \partial_{{\nu_{\bar g}}}|h|_{\bar g}^2 + \langle \delta_{\bar g} (h^2), {\nu_{\bar g}}\rangle_{\bar g} + 2 h( \nu_{\bar g}, \delta_{\bar g} h) + 2h({\nu_{\bar g}}, \nabla_{\bar g} tr_{\bar g} h )  + 2(tr_{\bar g} h) \langle \delta_{\bar g} h, {\nu_{\bar g}} \rangle_{\bar g}  \right] f d\sigma_{\bar g}\\
	&+  \int_\Sigma \left[ h^2({\nu_{\bar g}}, \nabla_{\bar g} f)  - |h|_{\bar g}^2 \partial_{{\nu_{\bar g}}} f  - 2 (tr_{\bar g} h)h( {\nu_{\bar g}}, \nabla_{\bar g} f) \right]  d\sigma_{\bar g},
	\end{align*}
	where $$\mathscr{R}_{\bar g}(h, h)  := \langle Rm_{\bar g} \cdot h, h \rangle_{\bar g}  + 2(Ric_{\bar g}\cdot h)(tr_{\bar g} h) -\frac{2 R_{\bar g}}{n-1}(tr_{\bar g} h)^2 .$$
\end{proposition}

\begin{proof}
	
	By Lemma \ref{lem:scalar_variation_formulae}, we have
	\begin{align*}
	\int_\Omega ( D^2 R_{\bar g} \cdot (h,h) )f dv_{\bar g}
	=& \int_{\Omega} \left[-2 \gamma_{\bar{g}} (h^2) - \Delta_{\bar g} |h|_{\bar g}^2 + 2 \langle h, \nabla^2_{\bar g} (tr_{\bar g} h) \rangle_{\bar g} + \nabla_{\alpha} h_{\beta\gamma} \nabla^{\beta} h^{\alpha\gamma} \right] f dv_{\bar g}\\
	& + \int_\Omega \left[- 2 \langle \delta_{\bar g} h, d(tr_{\bar g} h) \rangle_{\bar g} - \frac{1}{2} |\nabla_{\bar g} h|_{\bar g}^2 - \frac{1}{2}|d(tr_{\bar g} h)|_{\bar g}^2  \right] f dv_{\bar g}.
	\end{align*}
	From integration by parts,
	\begin{align*}
	&-2 \int_\Omega (\gamma_{\bar{g}} (h^2) ) f dv_{\bar{g}}\\ =& -2 \int_\Omega \langle\gamma^*_{\bar{g}} f, h^2\rangle_{\bar g} dv_{\bar{g}} -2  \int_\Sigma \left[ (tr_{\bar g} (h^2)) \partial_{{\nu_{\bar g}}}f - f \partial_{{\nu_{\bar g}}} (tr_{\bar g} (h^2)) -  h^2 ( {\nu_{\bar g}}, \nabla f)  - \langle\delta_{\bar g} (h^2), {\nu_{\bar g}}\rangle_{\bar g} f \right]d\sigma_{\bar g}\\
	=& -2 \int_\Omega \langle\gamma^*_{\bar{g}} f, h^2\rangle_{\bar g} dv_{\bar{g}} + 2  \int_\Sigma \left[  \left( \partial_{{\nu_{\bar g}}} |h|_{\bar g}^2  + \langle\delta_{\bar g} (h^2), {\nu_{\bar g}}\rangle_{\bar g} \right) f + h^2 ( {\nu_{\bar g}}, \nabla f) -|h|_{\bar g}^2 \partial_{{\nu_{\bar g}}}f \right] d\sigma_{\bar g}
	\end{align*}
	and
	\begin{align*}
	- \int_\Omega \left( \Delta_{\bar g} |h|^2 \right) fdv_{\bar{g}} = - \int_\Omega \left( |h|^2 \Delta_{\bar g} f \right) dv_{\bar{g}} - \int_\Sigma \left [ f \partial_{{\nu_{\bar g}}}|h|_{\bar g}^2 - |h|_{\bar g}^2 \partial_{{\nu_{\bar g}}} f \right] d\sigma_{\bar g}.
	\end{align*}
	Also,
	\begin{align*}
	&2 \int_\Omega \langle h, \nabla_{\bar g}^2 (tr_{\bar g} h) \rangle_{\bar g} f dv_{\bar{g}}\\
	=& 2 \int_\Omega \left[\langle \delta_{\bar g} h, d (tr_{\bar g}h) \rangle f - \langle h, d (tr_{\bar g} h) \otimes df \rangle_{\bar g} \right] dv_{\bar g} + 2 \int_\Sigma h({\nu_{\bar g}}, \nabla_{\bar g} (tr_{\bar g} h) ) f d\sigma_{\bar g}\\
	=& 2 \int_\Omega (tr_{\bar g} h) \left[ (\delta_{\bar g}^2 h) f - 2 \langle \delta_{\bar g} h, d f \rangle_{\bar g} + \langle h, \nabla_{\bar g}^2 f \rangle_{\bar g} \right] dv_{\bar g} \\
	&+ 2 \int_\Sigma \left[(h({\nu_{\bar g}}, \nabla_{\bar g} (tr_{\bar g} h) )  + (tr_{\bar g} h) \langle \delta_{\bar g} h, {\nu_{\bar g}} \rangle_{\bar g} ) f - (tr_{\bar g} h)h( {\nu_{\bar g}}, \nabla_{\bar g} f) \right]  d\sigma_{\bar g}\\
	=& 2 \int_\Omega (tr_{\bar g} h) \left[ (\delta_{\bar g}^2 h) f - 2 \langle \delta_{\bar g} h, d f \rangle_{\bar g} + \langle h, \gamma_{\bar g}^*f \rangle_{\bar g} + (tr_{\bar g} h) \Delta_{\bar g} f +  (Ric_{\bar g} \cdot h)  f  \right] dv_{\bar g} \\
	&+ 2 \int_\Sigma \left[(h({\nu_{\bar g}}, \nabla_{\bar g} (tr_{\bar g} h) )  + (tr_{\bar g} h) \langle \delta_{\bar g} h, {\nu_{\bar g}} \rangle_{\bar g} )f - (tr_{\bar g} h)h( {\nu_{\bar g}}, \nabla_{\bar g} f) \right]  d\sigma_{\bar g}
	\end{align*}
	and
	\begin{align*}
	&\int_\Omega \left[ \nabla_{\alpha} h_{\beta\gamma} \nabla^{\beta} h^{\alpha\gamma}\right] f dv_{\bar{g}} \\
	=&  -\int_\Omega h_\gamma^{\ \beta}  \left[ \nabla_{\alpha} \nabla_\beta h^{\alpha\gamma}f + \nabla_\beta  h^{\alpha\gamma}\nabla_\alpha f\right] dv_{\bar{g}} + \int_\Sigma \left[  h_{\beta\gamma}{\nu_{\bar g}}_{\alpha} \nabla^{\beta} h^{\alpha\gamma}\right] f d\sigma_{\bar g}\\
	=&  -\int_\Omega h_\gamma^{\ \beta}  \left[ (\nabla_\beta\nabla_{\alpha}  h^{\alpha\gamma} + R_{\alpha \beta \delta}^{\ \ \ \ \alpha} h^{\delta\gamma} + R_{\alpha \beta \delta}^{\ \ \ \ \gamma} h^{\alpha\delta})f + \nabla_\beta  h^{\alpha\gamma}\nabla_\alpha f\right] dv_{\bar{g}} + \int_\Sigma \left[  h_{\beta\gamma}{\nu_{\bar g}}_{\alpha} \nabla^{\beta} h^{\alpha\gamma}\right] f d\sigma_{\bar g}\\
	=&  -\int_\Omega \left[ - \nabla_\beta h_\gamma^{\ \beta} \nabla_{\alpha}  h^{\alpha\gamma}f - 2h_\gamma^{\ \beta} \nabla_{\alpha}  h^{\alpha\gamma} \nabla_\beta f - h_\gamma^{\ \beta}   h^{\alpha\gamma} \nabla_\beta\nabla_\alpha f + \left(\langle Ric_{\bar g}, h^2\rangle_{\bar g} - \langle Rm_{\bar g}\cdot h, h\rangle_{\bar g} \right)f \right] dv_{\bar{g}} \\
	&+ \int_\Sigma \left[  \left(h_{\beta\gamma}{\nu_{\bar g}}_{\alpha} \nabla^{\beta} h^{\alpha\gamma}  - h_\gamma^\beta {\nu_{\bar g}}_\beta \nabla_{\alpha}  h^{\alpha\gamma} \right)f - h_\gamma^\beta   h^{\alpha\gamma} {\nu_{\bar g}}_\beta\nabla_\alpha f\right]  d\sigma_{\bar g}\\
	=&  \int_\Omega \left[ |\delta_{\bar g} h|_{\bar g}^2 f - 2 \langle h, \delta_{\bar g} h \otimes df \rangle_{\bar g} + \langle \nabla_{\bar g}^2 f- f Ric_{\bar g}, h^2 \rangle_{\bar g} + \langle Rm_{\bar g} \cdot h, h \rangle_{\bar g} f \right] dv_{\bar g} \\
	&- \int_\Sigma \left[ \left( \langle \delta_{\bar g} (h^2), {\nu_{\bar g}}\rangle_{\bar g} -2 h( \nu_{\bar g}, \delta_{\bar g} h) \right)f + h^2({\nu_{\bar g}}, \nabla_{\bar g} f) \right]  d\sigma_{\bar g}\\
	=&  \int_\Omega \left[ |\delta_{\bar g} h|_{\bar g}^2 f - 2 \langle h, \delta_{\bar g} h \otimes df \rangle_{\bar g} + \langle \gamma_{\bar g}^*f + \bar g \Delta_{\bar g} f, h^2 \rangle_{\bar g} + \langle Rm_{\bar g} \cdot h, h \rangle_{\bar g} f \right] dv_{\bar g} \\
	&- \int_\Sigma \left[ \left( \langle \delta_{\bar g} (h^2), {\nu_{\bar g}}\rangle_{\bar g} -2 h( \nu_{\bar g}, \delta_{\bar g} h) \right)f + h^2({\nu_{\bar g}}, \nabla_{\bar g} f) \right]  d\sigma_{\bar g}.
	\end{align*}	
	Combining all calculations above, we obtain
	
	\begin{align*}
	&\int_\Omega ( D^2 R_{\bar g} \cdot (h,h) )f dv_{\bar g}\\
	=&  \int_\Omega \left[ - \frac{1}{2} |\nabla_{\bar g} h|_{\bar g}^2 - \frac{1}{2}|d(tr_{\bar g} h)|_{\bar g}^2 + |\delta_{\bar g} h|_{\bar g}^2 - 2 \langle \delta_{\bar g} h, d(tr_{\bar g} h) \rangle_{\bar g} + \langle Rm_{\bar g} \cdot h, h \rangle_{\bar g} + 2 (tr_{\bar g} h) (Ric_{\bar g} \cdot h)  \right] f  dv_{\bar g}\\
	&+\int_\Omega \left[ 2(tr_{\bar g} h) \left( (\delta_{\bar g}^2 h) f + \langle h, \gamma_{\bar g}^*f \rangle_{\bar g} - 2 \langle \delta_{\bar g} h, d f \rangle_{\bar g} + (tr_{\bar g} h) \Delta_{\bar g} f \right) - 2 \langle h, \delta_{\bar g} h \otimes df \rangle_{\bar g} - \langle\gamma^*_{\bar{g}} f, h^2\rangle_{\bar g}  \right] dv_{\bar g} \\
	&+ \int_\Sigma \left [   \left( \partial_{{\nu_{\bar g}}}|h|_{\bar g}^2 + \langle \delta_{\bar g} (h^2), {\nu_{\bar g}}\rangle_{\bar g} + 2 h( \nu_{\bar g}, \delta_{\bar g} h) \right)f  - |h|_{\bar g}^2 \partial_{{\nu_{\bar g}}} f + h^2({\nu_{\bar g}}, \nabla_{\bar g} f) \right] d\sigma_{\bar g}\\
	&+ 2 \int_\Sigma \left[(h({\nu_{\bar g}}, \nabla_{\bar g} (tr_{\bar g} h) )  + (tr_{\bar g} h) \langle \delta_{\bar g} h, {\nu_{\bar g}} \rangle_{\bar g} )f - (tr_{\bar g} h)h( {\nu_{\bar g}}, \nabla_{\bar g} f) \right]  d\sigma_{\bar g}\\
	=&  \int_\Omega \left[ - \frac{1}{2} |\nabla_{\bar g} h|_{\bar g}^2 - \frac{1}{2}|d(tr_{\bar g} h)|_{\bar g}^2 + |\delta_{\bar g} h|_{\bar g}^2 - 2 \langle \delta_{\bar g} h, d(tr_{\bar g} h) \rangle_{\bar g} + 2(tr_{\bar g} h) (\delta_{\bar g}^2 h)  + \mathscr{R}_{\bar g}(h, h)  \right] f  dv_{\bar g}\\
	&+\int_\Omega \left[ 2(tr_{\bar g} h) \left( \langle h, \gamma_{\bar g}^*f \rangle_{\bar g} - 2 \langle \delta_{\bar g} h, d f \rangle_{\bar g} - \frac{1}{n-1} (tr_{\bar g} h) \left(tr_{\bar g} (\gamma_{\bar g}^* f) \right) \right) - 2 \langle h, \delta_{\bar g} h \otimes df \rangle_{\bar g} - \langle\gamma^*_{\bar{g}} f, h^2\rangle_{\bar g} \right] dv_{\bar g} \\
	&+ \int_\Sigma \left [  \partial_{{\nu_{\bar g}}}|h|_{\bar g}^2 + \langle \delta_{\bar g} (h^2), {\nu_{\bar g}}\rangle_{\bar g} + 2 h( \nu_{\bar g}, \delta_{\bar g} h) + 2h({\nu_{\bar g}}, \nabla_{\bar g} (tr_{\bar g} h) )  + 2(tr_{\bar g} h) \langle \delta_{\bar g} h, {\nu_{\bar g}} \rangle_{\bar g}  \right] f d\sigma_{\bar g}\\
	&+  \int_\Sigma \left[ h^2({\nu_{\bar g}}, \nabla_{\bar g} f)  - |h|_{\bar g}^2 \partial_{{\nu_{\bar g}}} f  - 2 (tr_{\bar g} h)h( {\nu_{\bar g}}, \nabla_{\bar g} f) \right]  d\sigma_{\bar g},
	\end{align*}
	where we used the fact that 
	$$ tr_{\bar g} (\gamma_{\bar g}^* f) = - (n-1) \left( \Delta_{\bar g} f + \frac{R_{\bar g}}{n - 1} f \right) $$
	and
	$$\mathscr{R}_{\bar g}(h, h) = \langle Rm_{\bar g} \cdot h, h \rangle_{\bar g}  + 2(Ric_{\bar g}\cdot h)(tr_{\bar g} h) -\frac{2 R_{\bar g}}{n-1}(tr_{\bar g} h)^2 .$$	
\end{proof} 

\vskip 0.2in

In particular, for $V$-static metrics, we have the following identity:
\begin{corollary}\label{cor:V_static_scalar_curv_second_var} 
	Suppose $(\Omega, \bar g, f, \kappa)$ is a $V$-static space, then for any $h \in \ker \delta_{\bar g}$ with $h|_{T\Sigma} \equiv 0$,
	\begin{align*}
	&\int_\Omega \left( D^2R_{\bar g} \cdot (h, h) \right) f dv_{\bar{g}}  \\
	=& -\frac{1}{2} \int_\Omega \left[\left( |\nabla_{\bar g} h|_{\bar g}^2 + |d(tr_{\bar g} h)|_{\bar g}^2 - 2 \mathscr{R}_{\bar g}(h, h) \right)f +  2\kappa \left(  |h|_{\bar g}^2 + \frac{2}{n-1}(tr_{\bar g} h)^2 \right)\right] dv_{\bar{g}}\\ 
	&- \int_\Sigma \left[  A_{\bar g}^{ij} h_{in}h_{jn} - \left(  h_{nn}^2 - 3 \sum_{i=1}^{n-1} h^2_{in} \right) H_{\bar g} + 4 h_{nn} \left(  \nabla_i h_n^{\ i} - \frac{1}{2}  \nabla_n h_{i}^{\ i} \right)  \right] f  d\sigma_{\bar g}\\
	&- \int_\Sigma \left[   \left(2 h^2_{nn} + \sum_{i=1}^{n-1} h^2_{in}  \right) \partial_n f  + 2 h_{nn} \sum_{i=1}^{n-1} h_{in}\partial_i f  \right] d\sigma_{\bar g}.
	\end{align*}
\end{corollary}

\begin{proof}
	Applying Proposition \ref{prop:second_var_integral_scalar_curvature} with our assumptions, 
	\begin{align*}
	&\int_\Omega ( D^2 R_{\bar g} \cdot (h,h) )f dv_{\bar g}\\
	=&  - \frac{1}{2} \int_\Omega \left[ \left( |\nabla_{\bar g} h|_{\bar g}^2 + |d(tr_{\bar g} h)|_{\bar g}^2 -2 \mathscr{R}_{\bar g}(h, h)  \right) f  + 2 \kappa \left( |h|^2_{\bar g} + \frac{2}{n-1} (tr_{\bar g} h)^2\right)\right] dv_{\bar g} \\
	&+  \int_\Sigma \left[ \left (  \partial_{{\nu_{\bar g}}}|h|_{\bar g}^2 + \langle \delta_{\bar g} (h^2), {\nu_{\bar g}}\rangle_{\bar g} + 2h({\nu_{\bar g}}, \nabla_{\bar g} (tr_{\bar g} h) )  \right) f + h^2({\nu_{\bar g}}, \nabla_{\bar g} f)  - |h|_{\bar g}^2 \partial_{{\nu_{\bar g}}} f  - 2 (tr_{\bar g} h)h( {\nu_{\bar g}}, \nabla_{\bar g} f) \right]  d\sigma_{\bar g}.
	\end{align*}
	
	For the boundary integral, we will rewrite it in terms of the orthonormal frame chosen for the boundary. Note that we have the following identities
	\begin{align}
	\Gamma_{ij}^n = - A_{ij}^{\bar g},\quad \Gamma_{jn}^k = A_j^k, \quad \Gamma_{in}^i = H_{\bar g}
	\end{align}
	holds on $\Sigma$.
	Since $$ \delta_{\bar g} h = 0 $$ and $$h_{ij} = 0, \quad i,j = 1, \cdots, n-1,$$
	we have 
	$$\langle \delta_{\bar g} (h^2), {\nu_{\bar g}} \rangle_{\bar g} = (\delta_{\bar g} (h^2) )_n = - \nabla_\alpha (h_\beta^{\ \alpha}h_n^{\ \beta})= - h_\beta^{\ \alpha}\nabla_\alpha h_n^{\ \beta} =- h_{nn} \nabla_n h_{nn} - h_n^{\ i} \nabla_i h_{nn} - h_n^{\ i}\nabla_n h_{in}$$
	and 
	$$ \partial_{{\nu_{\bar g}}} |h|_{\bar g}^2 = \nabla_n |h|_{\bar g}^2 = 2 h_{nn} \nabla_n h_{nn} + 4 h_n^{\ i} \nabla_n h_{in}$$
	on $\Sigma$. Thus,
	\begin{align*}
	&\partial_{{\nu_{\bar g}}} |h|_{\bar g}^2 + \langle \delta_{\bar g} (h^2), {\nu_{\bar g}} \rangle_{\bar g} + 2 h( {\nu_{\bar g}}, \nabla_{\bar g} (tr_{\bar g} h) ) \\ 
	=&  h_{nn} \nabla_n h_{nn} + 3 h_n^{\ i} \nabla_n h_{in} -  h_n^{\ i} \nabla_i h_{nn} + 2 h_{nn} \nabla_n (tr_{\bar g} h) + 2 h_n^{\ i} \nabla_i (tr_{\bar g} h)\\
	=& 3h_{nn} \nabla_n h_{nn} + 3 h_n^{\ i} \nabla_n h_{in} - h_n^{\ i} \nabla_i h_{nn} + 2 h_{nn} \nabla_n h_i^{\ i} + 2 h_n^{\ i}  \nabla^\Sigma_i h_{nn}\\
	=& - 3h_{nn} \nabla_i h_n^{\ i} - 3 h_n^{\ i}  \nabla_j h_i^{\ j} - h_n^{\ i} \nabla_i h_{nn} + 2 h_{nn} \nabla_n h_i^{\ i} + 2 h_n^{\ i} \nabla^\Sigma_i h_{nn},
	\end{align*}
	where we used the fact
	$$\nabla_n h_{n \alpha} = - (\delta_{\bar g} h)_{\alpha} - \nabla_i h_\alpha^{\ i} = - \nabla_i h_\alpha^{\ i}.$$
	Moreover, from
	\begin{align*}
	\nabla_j h_i^{\ j} = \partial_j h_i^{\ j} + \Gamma_{j \alpha}^j h_i^{\ \alpha} - \Gamma_{ji}^\alpha h_\alpha^{\ j} = A_{ij}^{\bar g} h_n^{\ j} + H_{\bar{g}} h_{in}
	\end{align*}
	and
	\begin{align*}
	\nabla_i h_{nn} = \partial_i h_{nn} - 2 \Gamma_{in}^\alpha h_{\alpha n} = \nabla^\Sigma_i h_{nn} - 2A_{ij}^{\bar g} h_n^{\ j},
	\end{align*}
	we obtain
	\begin{align*}
	&\partial_{{\nu_{\bar g}}} |h|_{\bar g}^2 + \langle \delta_{\bar g} (h^2), {\nu_{\bar g}} \rangle_{\bar g} + 2 h( {\nu_{\bar g}}, \nabla_{\bar g} (tr_{\bar g} h) ) \\
	=& - A_{\bar g}^{ij} h_{in}h_{jn} - 3 H_{\bar g} \sum_{i=1}^{n-1} h^2_{in} + h_n^{\ i} \nabla^\Sigma_i h_{nn} - 3 h_{nn} \nabla_i h_n^{\ i} + 2 h_{nn} \nabla_n h_i^{\ i}.
	\end{align*}
	On the other hand,
	\begin{align*}
	& h^2( {\nu_{\bar g}} , \nabla_{\bar g} f ) - |h|_{\bar g}^2 \partial_{{\nu_{\bar g}}} f  - 2 (tr_{\bar g} h) h( {\nu_{\bar g}}, \nabla_{\bar g} f ) \\
	=& - \left( 2 h^2_{nn} + \sum_{i=1}^{n-1} h^2_{in}  \right) \partial_n f - h_{nn} \sum_{i=1}^{n-1} h_{in}\partial_i f.
	\end{align*}
	Applying integration by parts,
	\begin{align*}
	&\int_\Sigma \left[ \left (  \partial_{{\nu_{\bar g}}}|h|_{\bar g}^2 + \langle \delta_{\bar g} (h^2), {\nu_{\bar g}}\rangle_{\bar g} + 2h({\nu_{\bar g}}, \nabla_{\bar g} (tr_{\bar g} h) )  \right) f + h^2({\nu_{\bar g}}, \nabla_{\bar g} f)  - |h|_{\bar g}^2 \partial_{{\nu_{\bar g}}} f  - 2 (tr_{\bar g} h)h( {\nu_{\bar g}}, \nabla_{\bar g} f) \right]  d\sigma_{\bar g}\\
	=& -\int_\Sigma \left[\left(  A_{\bar g}^{ij} h_{in}h_{jn} + 3 H_{\bar g} \sum_{i=1}^{n-1} h^2_{in} \right)f  + \left( 2 h^2_{nn} + \sum_{i=1}^{n-1} h^2_{in}  \right) \partial_n f + 2 h_{nn} \sum_{i=1}^{n-1} h_{in}\partial_i f  \right] d\sigma_{\bar g}\\ 
	&\ \ + \int_\Sigma \left( - h_{nn} \nabla^\Sigma_i h_n^{\ i} - 3 h_{nn} \nabla_i h_n^{\ i} + 2 h_{nn} \nabla_n h_{i}^{\ i} \right) f  d\sigma_{\bar g} .
	\end{align*}
	Note that
	\begin{align*}
	\nabla_i h_n^{\ i} =\partial_i h_n^{\ i} + \Gamma_{i \alpha}^i h_n^{\ \alpha} - \Gamma_{in}^\alpha h_\alpha^{\ i} = \nabla_i^\Sigma h_n^{\ i} + H_{\bar{g}} h_{nn}
	\end{align*} 
	and hence
	\begin{align*}
	&\int_\Sigma \left[ \left (  \partial_{{\nu_{\bar g}}}|h|_{\bar g}^2 + \langle \delta_{\bar g} (h^2), {\nu_{\bar g}}\rangle_{\bar g} + 2h({\nu_{\bar g}}, \nabla_{\bar g} (tr_{\bar g} h) )  \right) f + h^2({\nu_{\bar g}}, \nabla_{\bar g} f)  - |h|_{\bar g}^2 \partial_{{\nu_{\bar g}}} f  - 2 (tr_{\bar g} h)h( {\nu_{\bar g}}, \nabla_{\bar g} f) \right]  d\sigma_{\bar g}\\
	=& - \int_\Sigma \left[  A_{\bar g}^{ij} h_{in}h_{jn} - \left(  h_{nn}^2 - 3 \sum_{i=1}^{n-1} h^2_{in} \right) H_{\bar g} + 4 h_{nn} \left(  \nabla_i h_n^{\ i} - \frac{1}{2}  \nabla_n h_{i}^{\ i} \right)  \right] f  d\sigma_{\bar g}\\
	&- \int_\Sigma \left[  \left( 2 h^2_{nn} +\sum_{i=1}^{n-1} h^2_{in}  \right) \partial_n f  + 2 h_{nn} \sum_{i=1}^{n-1} h_{in}\partial_i f  \right] d\sigma_{\bar g}.
	\end{align*}
	Therefore, the conclusion follows.
\end{proof}

\vskip 0.2in

In particular, for a special class of $V$-static spaces, we have
\begin{corollary}\label{cor:int_scalar_curv_second_var_divergence_gauge} 
	Suppose $(M^n, \bar g)$ is a closed Einstein manifold with $$Ric_{\bar g} = (n-1) \lambda \bar g,$$ then for any $h \in S_{2,\bar g}^{_{TT}} (M) \oplus (C^\infty (M) \cdot \bar g)$,
	we have
	\begin{align*}
	\int_M ( D^2 R_{\bar g} \cdot (h,h) ) dv_{\bar g} 
	= - \frac{1}{2} \int_M \left( - \langle h, \Delta_E^{\bar g} h \rangle_{\bar g}  + \frac{n^2 - 2}{n^2}|d(tr_{\bar g} h)|_{\bar g}^2 -2 (n-1) \lambda |h|_{\bar g}^2  \right) dv_{\bar g}.
	\end{align*}
\end{corollary}

\begin{proof}
	According to the $V$-static equation (\ref{eqn:V_static}), it is obvious that Einstein manifold $(M^n, \bar g)$ is a $V$-static space with $f\equiv 1$ on $M$ and $\kappa = - (n-1) \lambda$. By Corollary \ref{cor:V_static_scalar_curv_second_var}, we obtain
		\begin{align*}
		&\int_M ( D^2 R_{\bar g} \cdot (h,h) ) dv_{\bar g}\\
		=&  \int_M \left[ - \frac{1}{2} |\nabla_{\bar g} h|_{\bar g}^2 - \frac{1}{2}|d(tr_{\bar g} h)|_{\bar g}^2 + |\delta_{\bar g} h|_{\bar g}^2 + \mathscr{R}_{\bar g}(h, h) +  2 \lambda (tr_{\bar g} h)^2 + (n-1) \lambda |h|_{\bar g}^2   \right] dv_{\bar g}.
		\end{align*}
		From our assumption,
		$$\delta_{\bar g} h = - \frac{1}{n} d(tr_{\bar g} h) $$
		and hence
		\begin{align*}
		&\int_M ( D^2 R_{\bar g} \cdot (h,h) ) dv_{\bar g}\\
		=&  \int_M \left[ - \frac{1}{2} |\nabla_{\bar g} h|_{\bar g}^2 - \frac{n^2 - 2}{2n^2}|d(tr_{\bar g} h)|_{\bar g}^2 + \mathscr{R}_{\bar g}(h, h) + 2 \lambda (tr_{\bar g} h)^2 + (n-1) \lambda |h|_{\bar g}^2  \right] dv_{\bar g}.
		\end{align*}
		Since
		\begin{align*}
			\mathscr{R}_{\bar g}(h, h) =& \langle Rm_{\bar g} \cdot h, h \rangle_{\bar g}  + 2(Ric_{\bar g}\cdot h)(tr_{\bar g} h) -\frac{2 R_{\bar g}}{n-1}(tr_{\bar g} h)^2\\
			=& \langle Rm_{\bar g} \cdot h, h \rangle_{\bar g} -2 \lambda (tr_{\bar g} h)^2,
		\end{align*}
		thus
		\begin{align*}
		\int_M ( D^2 R_{\bar g} \cdot (h,h) ) dv_{\bar g} 
		=& - \frac{1}{2} \int_M \left( - \langle h, \Delta_E^{\bar g} h \rangle_{\bar g}  + \frac{n^2 - 2}{n^2}|d(tr_{\bar g} h)|_{\bar g}^2 -2 (n-1) \lambda |h|_{\bar g}^2  \right) dv_{\bar g}.
		\end{align*}				
\end{proof}

\vskip 0.2in

\section{Volume comparison for $V$-static spaces}

In this section, we will investigate the volume comparison for geodesic balls in generic $V$-static spaces. \\

Let $\Omega$ be an $n$-dimensional compact domain in a $V$-static space $(M^n, \bar g, f, \kappa)$ with $C^1$-boundary $\Sigma:=\partial \Omega$. We define the functional
\begin{align}\label{the functional}
\mathscr{F}_{\Omega, \bar g} [g] := \int_{\Omega}  R(g) f dv_{\bar{g}} + 2 \int_\Sigma H(g) f d\sigma_{\bar g} - 2 \kappa V_{\Omega}(g),
\end{align}
where 
\begin{align*}
	g \in \mathcal{M}_{\Omega, \Sigma, \bar g}:= \{ g \in \mathcal{M}_{\Omega} : g|_{T\Sigma} = \bar g|_{T\Sigma} \}
\end{align*}
is a Riemannian metric on $\Omega$ that induces the same metric with $\bar g$ on the boundary $\Sigma$ . \\

This functional is particularly designed for a given $V$-static space. The information of both volume and curvature is encoded in this single functional. It has excellent variational properties:
\begin{proposition}\label{prop:F_critical_pt}
The $V$-static metric $\bar g$ is a critical point of the functional $\mathscr{F}_{\Omega, \bar g} [g]$. That is, 
\begin{align}
	D \mathscr{F}_{\Omega, \bar g} \cdot h = 0, 
\end{align}
for any $h\in S_2(\Omega)$ with $h|_{T{\partial \Omega}} \equiv 0$.
\end{proposition}

\begin{proof}
Applying Proposition \ref{prop:int_scalar_first variation} and together with Lemma \ref{lem:mean_curv_variations} and \ref{lem:vol_variations},
\begin{align*}
D\mathscr{F}_{\Omega, \bar{g}} \cdot h 
=& \int_\Omega (DR_{\bar{g}} \cdot h) f dv_{\bar{g}} + 2 \int_{\partial \Omega} (D H_{\bar{g}} \cdot h) f d\sigma_{\bar{g}} - 2 \kappa \left(DV_{\Omega, \bar g} \cdot h\right)\\
=& \int_\Omega \left[\langle h, \gamma^*_{\bar{g}}f \rangle_{\bar g} - \kappa (tr_{\bar g} h) \right] dv_{\bar{g}} \\
&+ \int_{\partial \Omega} \left[ -  (\partial_n (tr_{\bar g} h)  +  (\delta_{\bar g} h)_n + 2\nabla_i h_n^{\ i} - \nabla_n h_i^{\ i} -  h_{nn} H_{\bar{g}}) f   - h_n^{\ i}\partial_i f\right]d\sigma_{\bar{g}},
\end{align*}
where we used the fact $tr_{\bar g} h = h_{nn}$ on ${\partial \Omega}$. Since
\begin{align*}
\nabla_i h_n^{\ i} = \partial_i h_n^{\ i}+ \Gamma_{i\alpha}^i h_n^{\ \alpha} - \Gamma_{i n}^\alpha h_\alpha^{\ i} = \nabla_i^\Sigma h_n^{\ i} + H_{\bar g} h_{nn},
\end{align*}
we have
\begin{align*}
(\delta_{\bar g} h )_n = - \nabla_\alpha h_n^{\ \alpha} = - \nabla_i^\Sigma h_n^{\ i} - \nabla_n h_{nn} - H_{\bar{g}} h_{nn}.
\end{align*}
Therefore,
\begin{align*}
D\mathscr{F}_{\Omega, \bar{g}} \cdot h =\int_\Omega \langle h, \gamma^*_{\bar{g}}f - \kappa \bar g \rangle_{\bar g}  dv_{\bar{g}} -  \int_{\partial \Omega} \left[ (\nabla_i^\Sigma h_n^{\ i}) f  + h_n^{\ i}\partial_i f \right]d\sigma_{\bar{g}} = - \int_{\partial \Omega} \nabla_i^\Sigma (h_n^{\ i} f) d\sigma_{\bar{g}} = 0.
\end{align*}
i.e. $\bar{g}$ is a critical point of $\mathscr{F}_{\Omega, \bar g}[g]$.
\end{proof}
\vskip 0.2in

For the second variation, it is straightforward from Lemma \ref{lem:mean_curv_variations} and \ref{lem:vol_variations} together with Corollary \ref{cor:V_static_scalar_curv_second_var}:
\begin{proposition}\label{prop:F_second_var}
	For any $h \in \ker \delta_{\bar g}$ with $h|_{T\Sigma} \equiv 0$, we have
	\begin{align*}
	&D^2\mathscr{F}_{\Omega, \bar g} \cdot ( h, h)  \\
	= & -\frac{1}{2} \int_\Omega \left[ \left( |\nabla_{\bar g} h|_{\bar g}^2 + |d(tr_{\bar g} h)|_{\bar g}^2 - 2 \mathscr{R}_{\bar g} (h , h) \right)f +  \frac{n+3}{n-1}(tr_{\bar g} h)^2 \kappa \right]  dv_{\bar g}  \\
	&- \int_\Sigma \left[\left(  A^{ij}_{\bar g} h_{in}h_{jn} - \frac{1}{2} \left( h_{nn}^2 - 2 \sum_{i=1}^{n-1} h^2_{in}   \right) H_{\bar g} + 2 h_{nn} \left(  \nabla_i h_n^{\ i} - \frac{1}{2}  \nabla_n h_{i}^{\ i} \right)  \right)f  \right] d\sigma_{\bar g}\\
	&- \int_\Sigma \left[ \left( 2 h^2_{nn} + \sum_{i=1}^{n-1} h^2_{in} \right) \partial_n f  + 2 h_{nn} \sum_{i=1}^{n-1} h_{in}\partial_i f  \right] d\sigma_{\bar g}. 
	\end{align*}
\end{proposition}
\vskip 0.2in

	In general, geometric functionals are invariant under actions of diffeomorphisms and it would cause degenerations on their second variations. In order to get rid of these degenerations, we need to find a metric modulo diffeomorphisms. This is usually referred to be \emph{gauge fixing} and it can be obtained by applying basic elliptic theory and implicit function theorem. For manifold with boundary, this can be achieved if one poses appropriate boundary conditions.
	\begin{lemma}[{\cite[Proposition 11]{B-M}}] \label{lem:slice} 
		Suppose $(\Omega^n, \bar g)$ is a compact Riemannian manifold with boundary. Fix a real number $p > n$, there exists a constant $\varepsilon_1 > 0$, such that for a metric $g$ on $\Omega$ with $$g|_{T\partial \Omega} = \bar g|_{T\partial \Omega}$$ and $$||g - \bar{g}||_{W^{2,p}({\Omega}, \bar{g})} < \varepsilon_1,$$ there exists a diffeomorphism $\varphi: {\Omega} \rightarrow {\Omega}$ such that $\varphi|_{_{\partial\Omega}} = \mathrm{id}$ and $h := \varphi^*g - \bar{g} \in \ker \delta_{\bar g}$. Moreover, $$||h||_{W^{2,p}({\Omega}, \bar{g})} \leq N ||g - \bar{g}||_{W^{2,p}({\Omega}, \bar{g})},$$ for some constant $N > 0$ that only depends on $(\Omega, \bar g)$.
	\end{lemma}
\vskip 0.2in

In particular, we take $\Omega$ to be a geodesic ball $B_r(p)$ at an interior point $p \in M$ with radius $r>0$. Then we have
\begin{proposition}\label{prop:V_static_fix_gauge}
	Suppose $(M^n, \bar g, \kappa, f)$ is a $V$-static space and $p \in M$ is an interior point. Then there is a constant $\varepsilon_1 > 0$ such that for any metric $g$ on $B_r(p)$ satisfies
	\begin{itemize}
		\item $R_g \geq R_{\bar g}$ \ \ in $B_r(p)$,
		\item $H_g \geq H_{\bar g}$ \ on $\partial B_r(p)$,
		\item $g|_{T\partial B_r(p)} = \bar g|_{T\partial B_r(p)}$,
		\item $||g - \bar{g}||_{C^2(B_r(p), \bar{g})} < \varepsilon_1$,
	\end{itemize}
	we can find a diffeomorphism $\varphi \in \mathscr D(B_r(p))$ such that $\varphi|_{\partial B_r(p)} = \mathrm{id}$ and
	\begin{align*}
		h:= \varphi^*g - \bar g \in \ker \delta_{\bar g} 
	\end{align*}
	satisfying $|h|_{\bar g} < \frac{1}{2}$ in $B_r(p)$, $h|_{T\partial B_r(p)} \equiv 0$ on $\partial B_r(p)$ and
	\begin{align*}
	||h||_{C^2(B_r(p), \bar{g})} \leq N ||g - \bar{g}||_{C^2(B_r(p), \bar{g})}		
	\end{align*}
	for some constant $N > 0$ depends only on $(B_r(p), \bar g)$. Additionally, we have
	\begin{itemize}
		\item $R_{\varphi^*g}  \geq R_{\bar g}$ \ \ in $B_r(p)$,
		\item $H_{\varphi^*g} \geq H_{\bar g}$ \ on $\partial B_r(p)$.
	\end{itemize}	
\end{proposition}

\begin{proof}
	The existence of constant $\varepsilon_1$ and diffeomorphism $\varphi$ is a straightforward application of Lemma \ref{lem:slice}. Furthermore, we have 
	\begin{itemize}
		\item $R_{\varphi^*g} = R_g \circ \varphi \geq R_{\bar g}$ \quad \quad \quad \quad \ in $B_r(p)$,
		\item $H_{\varphi^*g} = H_g \circ \varphi = H_g \geq H_{\bar g}$ \quad \ on $\partial B_r(p)$,
	\end{itemize}
	according to facts that the scalar curvature $R_{\bar g}$ is a constant on $M$ (see Remark \ref{rmk:scalar_curv_V_static}) and $\varphi|_{_{\partial B_r(p)}} = \mathrm{id}$.
\end{proof}
\vskip 0.2in

Let ${\hat g}_{_h} = \bar g + h$ be a metric on $B_r(p)$, where $h \in S_2(B_r(p))$ satisfies that $|h|_{\bar g} < \frac{1}{2}$ and $h|_{T\partial B_r(p)} \equiv 0$. From Proposition \ref{prop:F_critical_pt} and \ref{prop:F_second_var}, the remainder of expansion for $\mathscr{F}_{\Omega, \bar g}$ up to second order can be written as
\begin{align}\label{eqn:F_remainder}
r_{_{B_r(p),\bar g}} [h] := &\mathscr{F}_{B_r(p), \bar g} [{\hat g}_{_h}] - \mathscr{F}_{B_r(p), \bar g} [\bar g] - D\mathscr{F}_{B_r(p), \bar g} \cdot h - \frac{1}{2} D^2 \mathscr{F}_{B_r(p), \bar g} \cdot (h, h)\notag \\
=& \int_{B_r(p)} \left( R_{{\hat g}_{_h}} - R_{\bar g}\right) f dv_{\bar g} - 2 \kappa \left( V_{B_r(p)} ({\hat g}_{_h}) - V_{B_r(p)} (\bar g)\right) + I_{B_r(p)}[h] + I_{\partial B_r(p)}[h],
\end{align}
where
\begin{align*}
I_{B_r(p)}[h]:= \frac{1}{4} \int_{B_r(p)} \left[\left( |\nabla_{\bar g} h|_{\bar g}^2 + |d(tr_{\bar g} h)|^2 - 2 \mathscr{R}_{\bar g} (h , h) \right)f + \frac{n+3}{n-1}(tr_{\bar g} h)^2 \kappa\right] dv_{\bar g}
\end{align*}
and
\begin{align*}
&I_{\partial B_r(p)}[h]\\
:=&  \int_{\partial B_r(p)} \left[  2 \left(H_{{\hat g}_{_h}} - H_{\bar g} \right) + \frac{1}{2} A_{\bar g}^{ij} h_{in}h_{jn} - \frac{1}{4} \left( h_{nn}^2 - 2 \sum_{i=1}^{n-1} h^2_{in} \right) H_{\bar g} +  h_{nn} \left(  \nabla_i h_n^{\ i} - \frac{1}{2}  \nabla_n h_{i}^{\ i} \right)  \right] f d\sigma_{\bar g}\\
&+\int_{\partial B_r(p)} \left[ \left( h^2_{nn} + \frac{1}{2} \sum_{i=1}^{n-1} h^2_{in}   \right) \partial_n f  + h_{nn} \sum_{i=1}^{n-1} h_{in}\partial_i f \right] d\sigma_{\bar g}.
\end{align*}
The estimate for the reminder $r_{B_r(p), \bar g}[h]$ plays a key role in our proof. It mainly relies on estimates for lower bounds of integrals $I_{B_r(p)}$ and $I_{\partial B_r(p)}$.\\

The estimate for a lower bound of interior integral $I_{B_r(p)}$ is essentially due to the solution of the following variational problem:
\begin{align*} \label{L-laplace}
\mu (\Omega, \bar g) = \inf \left\{ \frac{\int_{\Omega} |\nabla_{\bar g} h|_{\bar g}^2 dv_{\bar{g}}}{\int_{\Omega} |h|_{\bar g}^2 dv_{\bar{g}}} \ :\ 
h\in S_2(\Omega),\ h \not\equiv 0 \text{ and } h|_{T\partial\Omega} \equiv 0\right\}.
\end{align*}
A basic estimate was obtained by Qing and the author in \cite[Lemma 3.7]{Q-Y_2}:
\begin{lemma}\label{lem:2_tensor_eigenvalue}
	Suppose $(M^n, \ \bar g)$ is a Riemannian manifold with dimension $n \geq 3$ and $B_r(p) $ is a geodesic ball of radius $r$ centered at any interior $p \in M$. Then, there are positive constants $\bar r$ and $c_0$ 
	such that
	\begin{equation}\label{lambda-ball}
	\mu (B_r(p), \bar g) \geq \frac {c_0}{r^2}
	\end{equation}
	for all $0 < r< \bar r$.
\end{lemma}

From this, we are readily to obtain an estimate for a lower bound of $I_{B_r(p)}$:
\begin{proposition}\label{prop:interior_est}
	Suppose $p \in M$ is an interior point with $f(p)>0$, then there is a constant $r_1 > 0$ such that
	\begin{align*}
	f(x) > 0
	\end{align*}
	for all $x \in \overline{B_{r_1}(p)} \subseteq M$. Furthermore, for all $r \in (0,r_1)$ and any $h\in S_2(B_r(p))$ with $h|_{T\partial B_r(p)} \equiv 0$, 
	\begin{align}
	I_{B_r(p)} [h] \geq \frac{1}{8} \left(\inf_{B_{r}(p)} f \right)  ||h||^2_{W^{1,2}(B_r(p), \bar g)}.
	\end{align}
\end{proposition}

\begin{proof}
	By continuity, we can choose a constant $r_1' > 0$ such that $f(x) > 0$ for all $x \in \overline{B_{r_1'}(p)}$. 
	
	It is straightforward that
	\begin{align*}
	|\mathscr{R}_{\bar g}(h,h) |= \left|\langle Rm_{\bar g}\cdot h, h\rangle_{\bar g} + 2 (Ric_{\bar g} \cdot h)(tr_{\bar g} h ) - \frac{2 R_{\bar g}}{n-1} (tr_{\bar g} h)^2 \right|\leq \Lambda_{r_1'} |h|_{\bar g}^2
	\end{align*}
	on $B_{r_1'}(p)$, where $\Lambda_{r_1'} = \Lambda (n, \bar g, ||Rm_{\bar g}||_{C^0(B_{r_1'}(p),\bar g)}) $ is a positive constant independent of $h$. Thus for any $r < r_1'$ and $h \in S_2(B_r(p))$ with $h|_{T\partial B_r(p)} \equiv 0$, we have
	\begin{align*}
	I_{B_r(p)} [h]
	\geq&  \frac{1}{4} \int_{B_r(p)} \left[\left( |\nabla_{\bar g} h|_{\bar g}^2 - 2 |\mathscr{R}_{\bar g} (h , h) |\right)f - 3n |\kappa|  |h|_{\bar g}^2 \right]  dv_{\bar g}\\
	\geq& \frac{1}{4} \int_{B_r(p)} \left[ \left(\inf_{B_{r}(p)} f \right) |\nabla_{\bar g} h|_{\bar g}^2  - \left(2 \Lambda_{r_1'}  \left(\sup_{B_{r}(p)} f \right) + 3n|\kappa| \right) |h|_{\bar g}^2  \right] dv_{\bar g}\\
	=&  \frac{1}{8} \left(\inf_{B_{r}(p)} f \right) ||h||^2_{W^{1,2}(B_r(p), \bar g)} + \frac{1}{8} \left(\inf_{B_{r}(p)} f \right)  \int_{B_r(p)}  \left[|\nabla_{\bar g} h|_{\bar g}^2  - \mu_r |h|_{\bar g}^2 \right]  dv_{\bar g},
	\end{align*}
	where
	\begin{align*}
		\mu_r :=& \frac{4 \Lambda_{r_1'}  \left(\sup_{B_{r}(p)} f \right) + \left(\inf_{B_{r}(p)} f \right) + 6n|\kappa| }{\inf_{B_{r}(p)} f }
		\leq \frac{ (4 \Lambda_{r_1'} + 1) \left(\sup_{B_{r_1'}(p)} f \right) + 6n|\kappa| }{\inf_{B_{r_1'}(p)} f }
		:= \bar \mu_{r_1'}.
	\end{align*}
	Applying Lemma \ref{lem:2_tensor_eigenvalue}, we can choose a positive constant $r_1 < r_1'$ sufficiently small such that   
	$$\int_{B_r(p)}  |\nabla_{\bar g} h|_{\bar g}^2  dv_{\bar g} \geq  \bar \mu_{r_1'} \int_{B_r(p)}  |h|_{\bar g}^2   dv_{\bar g}$$
	for all $r \in (0, r_1)$. Therefore,
	\begin{align*}
	I_{B_r(p)} [h]\geq \frac{1}{8} \left(\inf_{B_{r}(p)} f \right)  ||h||^2_{W^{1,2}(B_r(p), \bar g)}
	\end{align*}
	holds for any $r \in (0, r_1)$.
\end{proof}
\vskip 0.2in

For a lower bound estimate for the boundary integral $I_{\partial B_r(p)}$, we have
\begin{proposition}\label{prop:boundary_est}
Suppose $p \in M$ is an interior point with $f(p)>0$, then there is a constant $r_2 > 0$ such that
\begin{align*}
f(x) > 0
\end{align*}
for all $x \in \overline{B_{r_2}(p)} \subseteq M$. Furthermore, for all $r \in (0, r_2)$ and any metric ${\hat g}_{_h} := \bar g + h$ in $B_r(p)$ satisfies that 
\begin{itemize}
	\item $h \in S_2(B_r(p))$ with $|h|_{\bar g} < \frac{1}{2}$ and $h|_{T\partial B_r(p)} \equiv 0$,
	\item $H_{{\hat g}_{_h}} \geq H_{\bar g}$ \ on $\partial B_r(p)$,
\end{itemize}
then we have
\begin{align}
I_{\partial B_r(p)} [h]\geq  - C_0 \left(\sup_{B_r(p)} f\right) ||h||_{C^1(B_r(p), \bar g)} ||h||^2_{W^{1,2}( B_r(p), \bar g)},
\end{align}
where $C_0 > 0$ is a constant depends only on $(B_r(p), \bar g)$.
\end{proposition}

\begin{proof}
	By continuity, we can choose a constant $r_2' > 0$ such that $f(x) > 0$ for all $x \in \overline{B_{r_2'}(p)}$. 
	
	As observed in \cite{B-M}, for all $r \in (0, r_2')$ and any metric ${\hat g}_{_h} = \bar g + h$ satisfies that $h \in S_2(B_r(p))$ with $|h|_{\bar g} < \frac{1}{2}$ and $h|_{T\partial B_r(p)} \equiv 0$, we have
\begin{align*}
	h_{nn} (H_{{\hat g}_{_h}} - H_{\bar g}) 
	= \frac{1}{2} h_{nn}^2 H_{\bar g} -  h_{nn} \left(  \nabla_i h_n^{\ i} - \frac{1}{2}  \nabla_n h_{i}^{\ i} \right) + F_{\bar g} (h)
\end{align*}
due to Lemma \ref{lem:mean_curv_variations}, where the tail term $F_{\bar g} (h)$ satisfies that
\begin{align*}
	|F_{\bar g} (h)|_{\bar g} \leq \widetilde{C}_1 |h|_{\bar g}^2 (|\nabla_{\bar g} h|_{\bar g} + |A_{\bar g}|_{\bar g}|h|_{\bar g})
\end{align*}
and $\widetilde{C}_1 >0$ is a constant depends only on the dimension $n$. From this,
\begin{align*}
I_{\partial B_r(p)} [h] =& \int_{\partial B_r(p)} \left[ \left(2 - h_{nn}\right) \left(H_{\hat g} - H_{\bar g} \right) + \frac{1}{2}  A_{\bar g}^{ij} h_{in}h_{jn} + \frac{1}{4} \left(  h_{nn}^2 + 2 \sum_{i=1}^{n-1} h^2_{in}  \right) H_{\bar g} \right] f d\sigma_{\bar g}\\
&+ \int_{\partial B_r(p)} \left[ \left(  h^2_{nn} + \frac{1}{2} \sum_{i=1}^{n-1} h^2_{in} \right) \partial_n f  + h_{nn} \sum_{i=1}^{n-1} h_{in}\partial_i f  \right] d\sigma_{\bar g} + \widetilde F_{\bar g}(h), 
\end{align*}
where the tail term $\widetilde F_{\bar g}(h)$ satisfies that
\begin{align*}
	|\widetilde F_{\bar g}(h)| \leq& \widetilde{C}_2 \left(\sup_{B_r(p)} f\right) \int_{\partial B_r(p)} |h|_{\bar g}^2 \left(|\nabla_{\bar g} h|_{\bar g} + |A_{\bar g}|_{\bar g} |h|_{\bar g} \right) dv_{\bar g}
\end{align*}
for a constant $\widetilde{C}_2 > 0$ depends only on the dimension $n$.

For $r> 0$ sufficiently small, it is well-known that second fundamental form and mean curvature of the geodesic sphere $\partial B_r(p)$ behave similarly to round spheres in Euclidean space (see Exercise 1.123 in \cite{C-L-N}): $$A_{ij}^{\bar g} = \frac{1}{r}\bar g_{ij} + O(r)$$ and $$H_{\bar g} = \frac{n-1}{r} + O(r)$$
on $\partial B_r(p)$. Thus we can choose $r_2'' \in (0, r_2')$ such that
$$A_{ij}^{\bar g} \geq \frac{1}{2r}\bar g_{ij}$$ and $$H_{\bar g} \geq \frac{n-1}{2r}$$
holds for any geodesic sphere $\partial B_r(p)$ with $r < r_2''$.

For $r \in (0, r_2'')$, we have
\begin{align*}
&I_{\partial B_r(p)}[h]\\
\geq& \frac{1}{2} \int_{\partial B_r(p)} \left[ \frac{1}{4r} \left(   (n-1) h_{nn}^2 + 2n \sum_{i=1}^{n-1} h^2_{in}  \right)f  - \left(3 h_{nn}^2 + n  \sum_{i=1}^{n-1} h_{in}^2 \right) |\nabla_{\bar g} f|_{\bar g}\right] d\sigma_{\bar g}  + \widetilde F_{\bar g}(h) \\
=& \frac{1}{2} \int_{\partial B_r(p)} \left[ 3 \left(\frac{n-1}{12r} - \frac{|\nabla_{\bar g} f|_{\bar g}}{f}\right) h^2_{nn} + n\left( \frac{1}{2r}  - \frac{|\nabla_{\bar g} f|_{\bar g}}{f}\right)  \sum_{i=1}^{n-1} h_{in}^2 \right] f d\sigma_{\bar g}  + \widetilde F_{\bar g}(h).
\end{align*}
Since $f$ is positively lower bounded and $|\nabla_{\bar g} f|_{\bar g}$ is upper bonded on $B_{r_2''}(p)$, we can pick a constant $r_2 \in (0, r_2'')$ such that 
$$\frac{|\nabla_{\bar g} f|_{\bar g}}{f} \leq \min \left \{ \frac{n-1}{12r}, \frac{1}{2r} \right\}$$
holds in $B_r(p)$ for any $r \in (0, r_2)$ and hence
$$I_{\partial B_r(p)} \geq  \widetilde F_{\bar g}(h) \geq - \widetilde{C}_3 \left(\sup_{B_r(p)} f\right) ||h||_{C^1(\partial B_r(p), \bar g)}  ||h||^2_{L^2(\partial B_r(p), \bar g)}$$ holds for any $r \in (0, r_2)$, where $\widetilde{C}_3 > 0$ is a constant depends only on $n$ and $r$.

Recall the \emph{trace Sobolev's inequality}
\begin{align*}
 ||h||^2_{L^2(\partial B_r(p), \bar g)} \leq \theta_0 \ ||h||^2_{W^{1,2} (B_r(p), \bar g)},
\end{align*}
where $\theta_0 > 0$ is a constant depends only on $(B_r(p), \bar g)$. Therefore, we have the estimate
$$I_{\partial B_r(p)} \geq - C_0 \left(\sup_{B_r(p)} f\right) ||h||_{C^1(B_r(p), \bar g)} ||h||^2_{W^{1,2}( B_r(p), \bar g)}$$ holds for any $r \in (0, r_2)$, where $C_0:= \theta_0 \widetilde{C}_3  > 0$ is a constant depends only on $(B_r(p), \bar g)$. 
\end{proof}
\vskip 0.2in

Now we are readily to prove the main theorem in this section. 
\begin{proof}[Proof of Theorem A]
	Let $$r_0 := \min\{r_1, r_2\} > 0,$$ where $r_1$ and $r_2$ are given by Proposition \ref{prop:interior_est} and \ref{prop:boundary_est}. 
	
	For all $r \in (0, r_0)$, applying Proposition \ref{prop:V_static_fix_gauge}, we can find a constant $\varepsilon_1 > 0$ such that for any metric $g$ on $B_r(p) \subset M$ satisfies
	\begin{itemize}
		\item $R_g \geq R_{\bar g}$ in $B_r(p)$,
		\item $H_g \geq H_{\bar g}$ on $\partial B_r(p)$,
		\item $g|_{T\partial B_r(p)} = \bar g|_{T\partial B_r(p)}$,
		\item $||g - \bar g||_{C^2(B_r(p), \bar g)} < \varepsilon_1$,
	\end{itemize} 
	there is a diffeomorphism $\varphi \in \mathscr{D} (B_r(p))$ such that $\varphi|_{\partial B_r(p)} = \mathrm{id}$ and
	\begin{align*}
	h:= \varphi^*g - \bar g \in \ker \delta_{\bar g} 
	\end{align*}
	satisfying $|h|_{\bar g} < \frac{1}{2}$ in $B_r(p)$, $h|_{T\partial B_r(p)} \equiv 0$ on $\partial B_r(p)$ and
	\begin{align*}
	||h||_{C^2(B_r(p), \bar{g})} \leq N ||g - \bar{g}||_{C^2(B_r(p), \bar{g})}		
	\end{align*}
	for some constant $N > 0$ depends only on $(B_r(p), \bar g)$. Additionally, we have
	\begin{itemize}
		\item $R_{\varphi^*g}  \geq R_{\bar g}$ \ \ in $B_r(p)$,
		\item $H_{\varphi^*g} \geq H_{\bar g}$ \ on $\partial B_r(p)$.
	\end{itemize}
	
	Fix an $r\in (0,r_0)$ and we assume the contrary of the claimed volume comparison:
		\begin{align}\label{ineq:contrary_V-static_volume_comparison}
			\kappa(V_{B_r(p)}(g) - V_{B_r(p)} (\bar g)) \leq 0,
		\end{align}
	which implies
		\begin{align*}
			\kappa(V_{B_r(p)}(\varphi^* g) - V_{B_r(p)} (\bar g)) \leq 0.
		\end{align*}
	By Proposition \ref{prop:interior_est} and \ref{prop:boundary_est}, we have the following lower bound estimate for the remainder:
	\begin{align*}
	r_{_{B_r(p),\bar g}} [h] 
	= &\mathscr{F}_{B_r(p), \bar g} [\varphi^*g] - \mathscr{F}_{B_r(p), \bar g} [\bar g] - D\mathscr{F}_{B_r(p), \bar g} \cdot h - \frac{1}{2} D^2 \mathscr{F}_{B_r(p), \bar g} \cdot (h, h)\notag \\
	=& \int_{B_r(p)} \left( R_{\varphi^*g} - R_{\bar g}\right) f dv_{\bar g} - 2 \kappa \left( V_{B_r(p)} (\varphi^*g) - V_{B_r(p)} (\bar g)\right) + I_{B_r(p)}[h] + I_{\partial B_r(p)}[h]\\
	\geq& \left( \frac{1}{8} \left(\inf_{B_{r}(p)} f \right)  - C_0 \left(\sup_{B_r(p)} f\right) ||h||_{C^1(B_r(p), \bar g)} \right) ||h||^2_{W^{1,2}( B_r(p), \bar g)} .
	\end{align*}
	On the other hand, if we denote
	\begin{align*}
		\tau_r := \max \left\{ \sup_{B_r(p)} f,\ \sup_{B_r(p)} |\nabla_{\bar g} f|_{\bar g}\right\},
	\end{align*}
	then the upper bound of remainder can be estimated using Taylor's formula:
	\begin{align*}
	r_{_{B_r(p), \bar g}} [h] =& \frac{1}{6} D^3 \mathscr{F}_{B_r(p), \bar g + \xi h} \cdot (h, h, h)\\
	\leq& C_1 \tau_r \int_{B_r(p)} |h|_{\bar g}\left(|\nabla_{\bar g} h|_{\bar g}^2 + |h|_{\bar g}^2 \right)dv_{\bar g} +  C_2 \tau_r \int_{\partial B_r(p)} |h|_{\bar g}^2\left(|\nabla_{\bar g} h|_{\bar g} + |A_{\bar g}|_{\bar g}|h|_{\bar g} \right)dv_{\bar g}\\
	\leq& C_1 \tau_r ||h||_{C^0(B_r(p), \bar g)} ||h||^2_{W^{1,2}(B_r(p), \bar g)} + C_3 \tau_r ||h||_{C^1(B_r(p), \bar g)} ||h||^2_{L^2(\partial B_r(p), \bar g)}, 
	\end{align*}
	where $\xi \in (0,1)$ is a constant and $C_1, C_2, C_3$ are positive constants depends only on $(B_r(p), \bar g)$. Recall again the \emph{trace Sobolev inequality}
	\begin{align*}
	||h||^2_{L^2(\partial B_r(p), \bar g)} \leq \theta_0 \ ||h||^2_{W^{1,2} (B_r(p), \bar g)},
	\end{align*}
	where $\theta_0 > 0$ is constant depends only on $(B_r(p), \bar g)$. From this, we obtain
	\begin{align*}
	r_{_{B_r(p), \bar g}} [h] \leq& C_0'\tau_r ||h||_{C^1(B_r(p), \bar g)} ||h||^2_{W^{1,2}(B_r(p), \bar g)},
	\end{align*}
	where $C_0'= C_1 + \theta_0 C_3$ is a positive constant depends only on $(B_{r}(p), \bar g)$.
	
	Combining both lower and upper bound estimates of $r_{_{B_r(p), \bar g}}$, we obtain
	\begin{align}\label{ineq:V-static_rigidity}
		\left( \frac{1}{8} \left(\inf_{B_r(p)} f \right) - \left( C_0\left(\sup_{B_r(p)} f \right) + C_0'\tau_r \right) ||h||_{C^1(B_r(p), \bar g)} \right) ||h||^2_{W^{1,2}( B_r(p), \bar g)} \leq 0.
	\end{align}
	Take $$\varepsilon_0 := \frac{1}{N} \min \left\{\varepsilon_1,\ \frac{1}{8} \left( C_0\left(\sup_{B_r(p)} f \right) + C_0'\tau_r \right)^{-1} \left(\inf_{B_{r}(p)} f \right)  \right\} ,$$
	then for the metric $g$ satisfies
	\begin{align*}
		||g - \bar g||_{C^2(B_r(p), \bar g)} < \varepsilon_0,
	\end{align*}
	we have
	$$ ||h||_{C^1(B_r(p), \bar g)} \leq N ||g - \bar g||_{C^2(B_r(p), \bar g)} < N \varepsilon_0 < \frac{1}{8} \left( C_0\left(\sup_{B_r(p)} f \right) + C_0'\tau_r \right)^{-1} \left(\inf_{B_{r}(p)} f \right) .$$
	According to inequality (\ref{ineq:V-static_rigidity}), it implies $h$ vanishes identically on $B_r(p)$ and hence $\varphi^*g =  \bar g$, which shows that $\varphi: B_r(p) \rightarrow B_r(p)$ has to be an isometry. Therefore, the reverse of inequality (\ref{ineq:contrary_V-static_volume_comparison}) holds:
	\begin{align}
	\kappa(V_{B_r(p)}(g) - V_{B_r(p)} (\bar g)) \geq 0.
	\end{align}
	That is, the following volume comparison holds:
	\begin{itemize}
		\item if $\kappa < 0$, then $$V_{B_r(p)}(g) \leq V_{B_r(p)} (\bar g);$$
		\item if $\kappa > 0$, then $$V_{B_r(p)}(g) \geq V_{B_r(p)} (\bar g);$$
	\end{itemize} 
	with equality holds in either case if and only if the metric $g$ is isometric to $\bar g$.	
\end{proof}
\ \\


\section{Volume comparison for closed Einstein manifolds}

Suppose $(M^n, \bar g, f, \kappa)$ is closed $V$-static manifold, the functional $\mathscr{F}_{M, \bar g}$ introduced in the previous section can be simplified as
\begin{align}
\mathscr{F}_{M, \bar g}  [g] = \int_M R(g) f dv_{\bar g} - 2\kappa V_M (g).
\end{align}
According to Proposition \ref{prop:F_critical_pt}, the metric $\bar g$ is still a critical point of $\mathscr{F}_{M, \bar g}$. However, it is obvious that this functional is not compatible with actions of dilations, which would cause subtle issues in its second variation. Geometrically speaking, dilations introduce additional degeneracy besides actions of diffeomorphisms, since they make no essential change to the geometry of the manifold. In order to obtain volume comparison for closed manifolds, we need to construct a new functional instead, which is invariant under dilations.\\

\begin{definition}
	Suppose $(M^n, \bar g, f, \kappa)$ is an $n$-dimensional closed $V$-static manifold, we define the functional
	\begin{align}
	\mathscr{G}_{M, \bar g} [g] := \left(V_M (g) \right)^{\frac{2}{n}}\int_M R(g) f dv_{\bar g}
	\end{align}
	for any Riemannian metric $g$ on $M$.
\end{definition} 

	Obviously, this functional is dilation-invariant:
	$$\mathscr{G}_{M, \bar g} [c^2 g] = \left(V_M (c^2 g) \right)^{\frac{2}{n}}\int_M R(c^2 g) f dv_{\bar g}  = \mathscr{G}_{M, \bar g} [g]$$
	for any constant $c \neq 0$. \\

	Now we focus on a special type of $V$-static metrics: Einstein metrics. According to the $V$-static equation (\ref{eqn:V_static}), we get 
	$$\gamma_{\bar g}^* 1 = - Ric_{\bar g} = \kappa \bar g$$
	by taking the function $f$ to be constantly $1$ on $M$. This means, $(M^n, \bar g, 1, \kappa)$ is a $V$-static space if and only if the metric $\bar g$ is an Einstein metric with scalar curvature $R_{\bar g} = - n \kappa$. Moreover, if we denote
	\begin{align}
		\lambda := \frac{R_{\bar g}}{n(n-1)},
	\end{align}
	then the Ricci curvature tensor is given by
	\begin{align*}
	Ric_{\bar g} = (n-1)\lambda \bar g 
	\end{align*}
	and $$\kappa = - (n-1) \lambda.$$
\vskip 0.2in
	
	As a functional particularly designed for $V$-static metrics, $\mathscr G_{M,\bar g}$ shares analogous variational properties with $\mathscr F_{M,\bar g}$:

\begin{proposition}\label{prop:G_first_variations}
	Suppose $(M^n, \bar g)$ is a closed Einstein manifold with Ricci curvature tensor
	\begin{align*}
		Ric_{\bar g} = (n-1) \lambda \bar g,
	\end{align*}
	then the metric $\bar g$ is a critical point of the functional $\mathscr{G}_{M, \bar g}$.
\end{proposition}

\begin{proof}
	From Proposition \ref{prop:int_scalar_first variation} and Lemma \ref{lem:vol_variations}, 
	\begin{align*}
	&D\mathscr{G}_{M, \bar g} \cdot h \\
	=&\left(V_M (\bar g) \right)^{\frac{2}{n}} \int_M \left(DR_{\bar g} \cdot h \right) dv_{\bar g} + \frac{2}{n}\left(V_M (\bar g) \right)^{\frac{2}{n} - 1} (DV_{M,\bar g} \cdot h) \int_M R_{\bar g} dv_{\bar g} \\
	=& \left(V_M (\bar g) \right)^{\frac{2}{n}} \left[ \int_M (\gamma_{\bar g}^*1)\ dv_{\bar g} + \frac{1}{n}R_{\bar g} \int_M \left(tr_{\bar g} h \right) dv_{\bar g} \right]\\
	=& - \left(V_M (\bar g) \right)^{\frac{2}{n}} \int_M \left\langle Ric_{\bar g} - (n-1)\lambda \bar g, h  \right \rangle_{\bar g} dv_{\bar g} \\
	=& 0,
	\end{align*}
	for any $h \in S_2(M)$.
\end{proof}

\vskip 0.2in
	
For second variation, we have
\begin{proposition}\label{prop:G_second_variations}
	Suppose $(M^n, g)$ is an Einstein manifold with Ricci curvature tensor
	\begin{align*}
		Ric_{\bar g} = (n-1) \lambda \bar g,
	\end{align*}
	then 
	\begin{align*}
		&D^2\mathscr{G}_{M, \bar g} \cdot (h,h) \notag\\
		=&-\frac{1}{2} \left(V_M (\bar g) \right)^{\frac{2}{n}} \int_M \left[  - \langle h_{_{TT}}, \Delta^{\bar g}_E h_{_{TT}} \rangle_{\bar g} + \frac{(n-1)(n + 2)}{n^2} \left(|d(tr_{\bar g} h)|_{\bar g}^2  - n \lambda \left(tr_{\bar g} h - \overline{tr_{\bar g} h} \right)^2  \right)\right]  dv_{\bar g}
	\end{align*}
	for any $h = h_{_{TT}} + \frac{1}{n} (tr_{\bar g} h )\bar g \in S_{2,\bar g}^{_{TT}} \oplus (C^\infty(M) \cdot \bar g)$.
\end{proposition}	 

	\begin{proof}
		From Lemma \ref{lem:scalar_variation_formulae}, \ref{lem:vol_variations} and Corollary \ref{cor:int_scalar_curv_second_var_divergence_gauge}, we obtain
			\begin{align*}
			&D^2\mathscr{G}_{M, \bar g} \cdot (h,h) \\
			=& \frac{2}{n} \left(V_M(\bar g)\right)^{\frac{2}{n} - 1} \left(D^2V_{M,\bar g} \cdot (h, h)\right)\int_M R_{\bar g} dv_{\bar g} + \frac{4}{n} \left(V_M(\bar g)\right)^{\frac{2}{n} - 1} \left(DV_{M,\bar g} \cdot h\right)\int_M \left( DR_{\bar g} \cdot h \right) dv_{\bar g} \\
			&- \frac{2(n-2)}{n^2} \left(V_M(\bar g)\right)^{\frac{2}{n} - 2} \left(DV_{M,\bar g} \cdot h\right)^2 \int_M R_{\bar g}  dv_{\bar g} + \left(V_M(\bar g)\right)^{\frac{2}{n}} \int_M \left(D^2R_{\bar g} \cdot (h, h)\right) dv_{\bar g}\\
			=& - \frac{1}{2} \left(V_M(\bar g)\right)^{\frac{2}{n}} \int_M \left(  - \langle h, \Delta_E^{\bar g} h \rangle_{\bar g} + \frac{n^2 - 2}{n^2} |d(tr_{\bar g} h)|_{\bar g}^2 - (n-1)\lambda (tr_{\bar g} h)^2  \right) dv_{\bar g}\\
			&- \frac{(n-1)(n+2)}{2n}\lambda \left(V_M(\bar g)\right)^{\frac{2}{n}}  \int_M (\overline {tr_{\bar g} h})^2 dv_{\bar g}.
			\end{align*}
			Now the decomposition
			\begin{align*}
				h = h_{_{TT}} + \frac{1}{n}(tr_{\bar g} h) \bar g
			\end{align*} 
			implies
			\begin{align*}
			&(D^2\mathscr{G}_{M, \bar g}) \cdot (h,h) \\
			=&-\frac{1}{2} \left(V_M (\bar g) \right)^{\frac{2}{n}} \int_M \left[ - \langle h_{_{TT}}, \Delta^{\bar g}_E h_{_{TT}}  \rangle_{\bar g} + \frac{(n-1)(n + 2)}{n^2} \left(|d(tr_{\bar g} h)|_{\bar g}^2  - n \lambda \left(tr_{\bar g} h - \overline{tr_{\bar g} h} \right)^2  \right)\right]  dv_{\bar g}.
			\end{align*}			
	\end{proof}
	
	As a key step of the proof for our volume comparison theorem, we need to give a characterization of the second variation of the functional $\mathscr{G}_{M,\bar g}$ at $\bar g$. This is closely related to spectrum problems of two operators: one is about the Einstein operator and it can be characterized by the stability of Einstein metrics, the other one is about the Laplace-Beltrami operator and its eigenvalue estimate is given by the well-known \emph{Lichnerowicz-Obata theorem} (cf. Theorem 5.1 in \cite{Li}).\\

\begin{lemma}[Lichnerowicz-Obata's eigenvalue estimate]\label{lem:Lich_Obata_eigenvalue_est}
	Suppose $(M^n, \bar g)$ is an $n$-dimensional closed Riemannian manifold with Ricci curvature tensor
	$$Ric_{\bar g} \geq (n-1) \lambda \bar g,$$
	where $\lambda > 0$ is a constant. Then for any function $u \in C^\infty(M)$ that is not identically a constant, we have
	\begin{align}
	\int_M |d u|^2 dv_{\bar g} \geq n \lambda \int_M (u - \overline u)^2 dv_{\bar g},
	\end{align}
	where equality holds if and only if $(M^n, \bar g)$ is isometric to the round sphere $\mathbb{S}^n \left( r \right)$ with radius $r = \frac{1}{\sqrt{\lambda}}$ and $u$ is a first eigenfunction of the Laplace-Beltrami operator.
\end{lemma}
\vskip 0.2in

Applying it to Proposition \ref{prop:G_second_variations}, immediately we get the non-positive definite property of the second variation of $\mathscr{G}_{M,\bar g}$ at $\bar g$:	
\begin{proposition}\label{prop:Einstein_second_var_non_positive_def}
	Suppose $(M^n, \bar g)$ is a closed stable Einstein manifold with Ricci curvature tensor
	\begin{align*}
	Ric_{\bar g} = (n-1) \lambda \bar g,
	\end{align*}
	then
	\begin{align*}
	&D^2\mathscr{G}_{M, \bar g} \cdot (h,h) \leq 0
	\end{align*}
	holds for any $h \in S_{2,\bar g}^{_{TT}}(M) \oplus ( C^\infty (M) \cdot \bar g)$. Moreover, equality holds if and only if
	\begin{itemize}
		\item $h \in \mathbb{R} \bar g \oplus  \ker \Delta_E^{\bar g} $, when $(M, \bar g)$ is not isometric to the round sphere up to a rescaling of the metric;
		\item $h \in (\mathbb{R}  \oplus E_{n\lambda}) \bar g $, when $(M, \bar g)$ is isometric to the round sphere $\mathbb{S}^n(r)$ with radius $r = \frac{1}{\sqrt \lambda}$,
	\end{itemize}
	where $$E_{n\lambda} := \{ u \in C^\infty(\mathbb{S}^n(r)): \Delta_{\mathbb{S}^n(r)} u + n\lambda u = 0 \}$$ is the space of first eigenfunctions for the spherical metric.\\
\end{proposition}
\begin{proof}
	Recall that the Einstein metric $\bar g$ is stable if and only if $(-\Delta_E^{\bar g})$ is a non-negative operator. Then the conclusion follows by applying this fact and Lemma \ref{lem:Lich_Obata_eigenvalue_est} to Proposition \ref{prop:G_second_variations}.
\end{proof}
\vskip 0.2in

Intuitively speaking, a \emph{slice} is a subset of metrics in the space of all Riemannian metrics, which is transverse to the orbit of diffeomorphism actions. The following refined version of slice theorem reveals the local structure of Einstein metrics in the space of all metrics. It seems not appearing in references before to the best of the author's knowledge and we hope it can be useful in problems involving Einstein metrics. The proof is standard, please refer to \cite{B-M, Viaclovsky}.
\begin{theorem}[Ebin-Palais slice theorem] \label{thm:slice_thm}
	Suppose $(M^n,\bar{g})$ is a closed $n$-dimensional Einstein manifold with Ricci curvature tensor
	\begin{align*}
		Ric_{\bar g} = (n-1) \lambda \bar g,
	\end{align*}
	where $\lambda \in \mathbb{R}$ is a constant. Let $\mathcal{M}$ be the space of all Riemannian metrics on $M$. There exists a local slice $\mathcal{S}_{\bar{g}}$ though $\bar g$ in $\mathcal{M}$. That is, for a fixed real number $p > n$, one can find a constant $\varepsilon_1 > 0$ such that for any metric $g \in \mathcal{M}$ with $||g-\bar{g}||_{W^{2,p}(M,\bar{g})} < \varepsilon_1$, there is a diffeomorphism $\varphi\in \mathscr{D}(M)$ with $\varphi^*g \in \mathcal{S}_{\bar{g}}$. Moreover, for a smooth local slice $\mathcal{S}_{\bar{g}}$, we have the decomposition
	$$S_2(M) =T_{\bar g} \mathcal{S}_{\bar g} \oplus (T_{\bar g} \mathcal{S}_{\bar g})^{\perp},$$ where the tangent space of $\mathcal{S}_{\bar g}$ at $\bar g$ and its $L^2$-orthogonal complement are given by
	$$T_{\bar g}\mathcal{S}_{\bar g} = S_{2,\bar g}^{_{TT}}(M) \oplus (C^\infty (M) \cdot \bar g)$$
	and 
	$$(T_{\bar g}\mathcal{S}_{\bar g})^{\perp}= \left\{\mathcal{L}_{\bar g}(X)\ | \ \langle X, \nabla_{\bar g} u \rangle_{L^2(M, \bar g)} = 0, \ \forall u\in C^{\infty}(M) \right\},$$
	when $(M^n, \bar g)$ is not isometric to the round sphere $\mathbb{S}^{n}(r)$ up to a scaling;
	$$T_{\bar g}\mathcal{S}_{\bar g} =  S_{2,\bar g}^{_{TT}}(M) \oplus (E_{n\lambda}^\perp  \cdot \bar g)$$
	and 
	$$(T_{\bar g} \mathcal{S}_{\bar g})^{\perp}= \left\{\mathcal{L}_{\bar g}(X)\ | \ \langle X, \nabla_{\bar g} u \rangle_{L^2(M, \bar g)} = 0, \ \forall u\in E_{n\lambda}^\perp \right\},$$
	when $(M^n,\bar{g})$ is isometric to the round sphere $\mathbb{S}^{n}(r)$ with $r=\frac{1}{\sqrt{\lambda}}$. Here
	$$E_{n\lambda} = \{ u \in C^\infty(\mathbb{S}^n(r)): \Delta_{\mathbb{S}^n(r)} u + n\lambda u = 0 \}$$ is the space of first eigenfunctions for the spherical metric.
\end{theorem}
\vskip .1in

Now we restrict the functional $\mathscr{G}_{M,\bar g}$ on a local slice $\mathcal{S}_{\bar g}$ and denote it to be
\begin{align*}
	\mathscr{G}_{M,\bar g}^{\mathcal S}: = \mathscr{G}_{M,\bar g}|_{_{\mathcal S}}.
\end{align*}
In order to investigate the local behavior of $\mathscr{G}_{M,\bar g}^{\mathcal S}$ near $\bar g$, we need the following \emph{Morse lemma on Banach manifold}:
\begin{lemma}[Morse lemma \cite{F-M}]\label{lem:Morse_lemma}
	Let $\mathcal{P}$ be a Banach manifold and $F : \mathcal{P} \rightarrow \mathbb{R}$ a $C^2$-function. Suppose $\mathcal{Q} \subset \mathcal{P}$ is a submanifold, $F = 0$ and $dF = 0$ on $\mathcal{Q}$ and that there is a smooth normal bundle neighborhood of $\mathcal{Q}$ such that if $\mathcal{E}_x$ is the normal complement to $T_x \mathcal{Q}$ in $T_x \mathcal{P}$ then $d^2 F (x)$ is weakly negative definite on $\mathcal{E}_x$ (\emph{i.e.} $d^2 F (x) (v,v) \leq 0$ with equality only if $v = 0$). Let $\langle, \rangle_x$ be a weak Riemannian structure with a smooth connection and assume that $F$ has a smooth $\langle, \rangle_x$-gradient, $Y(x)$. Assume $DY(x)$ maps $\mathcal{E}_x$ to $\mathcal{E}_x$ and is an isomorphism for $x \in \mathcal{Q}$. Then there is a neighborhood $U$ of $\mathcal{Q}$ such that $y \in U$, $F(y) \geq 0$ implies $y \in \mathcal{Q}$.
\end{lemma}
\vskip 0.2in

	Applying it to our case, we obtain the following local rigidity result:
\begin{proposition}\label{prop:Einstein_rigidity}
	Suppose $(M^n, \bar g)$ is a strictly stable Einstein manifold and $\mathcal{S}_{\bar g}$ is a local slice through $\bar g$. There is a neighborhood $U_{\bar g}$ of $\bar g$ in $\mathcal{S}_{\bar g}$ such that for any metric $\hat g_{_S} \in U_{\bar g}$ satisfies
	\begin{align*}
		\mathscr{G}_{M,\bar g}^{\mathcal S} [\hat g_{_S}] \geq \mathscr{G}_{M,\bar g}^{\mathcal S} [\bar g],
	\end{align*}
	there is a constant $c>0$ such that $\hat g_{_S} = c^2 \bar g$. 
\end{proposition}

\begin{proof}
		Let $$\widetilde{\mathcal{Q}}_{\bar g} : = \{ g_{_S} \in \mathcal S_{\bar g} : g_{_S} \ is\ Einstein.\}$$
		be the subset of the local slice $\mathcal{S}_{\bar g}$ consisted of Einstein metrics near the reference metric $\bar g$. By a result of Koiso (Corollary 3.4 in \cite{Koiso}), strict stability implies that $\bar g$ is rigid. That is, we can find a neighborhood $\widetilde U_{\bar g} \subseteq \mathcal{S}_{\bar g}$ of $\bar g$ such that
		$$\mathcal{Q}_{\bar g} := \widetilde{\mathcal{Q}}_{\bar g} \cap \widetilde{U}_{\bar g} = \{ g_{_S} \in \widetilde{U}_{\bar g} :\ g_{_S} = c^2 \bar g,\ c > 0 \}.$$
		In particular, the tangent space of $\mathcal{Q}_{\bar g}$ at $\bar g$ is given by
		$$T_{\bar g} \mathcal Q_{\bar g} = \mathbb{R} \bar g$$
		 and its $L^2$-orthogonal complement in $T_{\bar g} \mathcal{S}_{\bar g}$ can be expressed as 
		 $$\mathcal{E}_{\bar g} : = (T_{\bar g} \mathcal Q_{\bar g})^\perp = S_{2, \bar g}^{_{TT}} (M) \oplus ( \Psi_{\bar g}(M) \cdot \bar g)$$
		 due to Theorem \ref{thm:slice_thm}, where 
		 $$ \Psi_{\bar g}(M) = \left\{ u \in E_{n\lambda}^\perp : \int_M u\ dv_{\bar g} = 0 \right\}$$
		 if $\bar g$ is spherical and
		 $$ \Psi_{\bar g}(M) = \left\{ u \in C^\infty (M) : \int_M u\ dv_{\bar g} = 0 \right\}$$
		 otherwise.
		
		Consider a weak Riemannian structure on the local slice $\mathcal{S}_{\bar g}$,
		\begin{align*}
			\langle\langle \cdot , \cdot \rangle\rangle_{g_{_S}} : \ T_{g_{_S}} \mathcal{S}_{\bar g} \times T_{g_{_S}} \mathcal{S}_{\bar g} \rightarrow \mathbb{R}, \qquad \forall g_{_S} \in \mathcal{S}_{\bar g},
		\end{align*}
		which is defined to be
		\begin{align*}
			\langle\langle h, k \rangle\rangle_{g_{_S}} := \int_M \left[ \langle \nabla_{g_{_S}} h, \nabla_{g_{_S}} k\rangle_{g_{_S}} + \langle h, k\rangle_{g_{_S}} \right] dv_{g_{_S}}  = \int_M \langle (-\Delta_{g_{_S}} + 1) h, k\rangle_{g_{_S}}  dv_{g_{_S}}
		\end{align*}
		for any $h, k \in T_{g_{_S}} \mathcal{S}_{\bar g}$. According to \cite{Ebin}, it has a smooth connection. The $\langle\langle \ ,\  \rangle\rangle_{g_{_S}}$-gradient of $\mathscr{G}_{M, \bar g}^{\mathcal S}$ is given by 
		\begin{align*}
		Y(g_{_S}) =& P_{g_{_S}} ( - \Delta_{g_{_S}} + 1)^{-1} \left[ (V_M(g_{_S}))^{\frac{2}{n}}\left(\gamma_{g_{_S}}^*f_{g_{_S}} + \frac{1}{n} {g_{_S}} \left(V_{M}(g_{_S})\right)^{- \frac{n+2}{n}} \mathscr{G}_{M,\bar g}[{g_{_S}}] \right)\right],
		\end{align*} 
		where $P_{g_{_S}}$ is the orthogonal projection on $T_{g_{_S}} \mathcal{S}_{\bar g}$ and $f_{g_{_S}}$ is a smooth function on $M$ such that $dv_{\bar g} = f_{g_{_S}} dv_{g_{_S}}$. Obviously, $Y(g_{_S})$ is a smooth vector field on $\mathcal{S}_{\bar g}$. For simplicity, we denote 
		\begin{align*}
			Z({g_{_S}}): =  (V_M(g_{_S}))^{\frac{2}{n}} \left(\gamma_{g_{_S}}^*f_{g_{_S}} + \frac{1}{n} {g_{_S}} \left(V_{M}(g_{_S})\right)^{- \frac{n+2}{n}} \mathscr{G}_{M,\bar g}[{g_{_S}}] \right). 
		\end{align*}
		It is straightforward that $Z(\bar g) = 0$ and the linearization of $Z$ at $\bar g$ is given by
		\begin{align*}
		(DZ_{\bar g}) \cdot h =&\frac{1}{2} \left(V_{M}(\bar g)\right)^{\frac{2}{n}} \left( \Delta^{\bar g}_E h_{_{TT}}  + \frac{(n-1)(n+2)}{n^2} \bar g \left(\Delta_{\bar g} +  n\lambda \right)(tr_{\bar g} h - \overline{tr_{\bar g} h}) \right)\\
		=& D^2 \mathscr{G}_{M,\bar g}\cdot (h, \cdot)
		\end{align*}
		for any $h = h_{_{TT}} + \frac{1}{n} (tr_{\bar g} h) \bar g \in \mathcal{E}_{\bar g}$. Thus, 
		\begin{align*}
			(DY_{\bar g}) \cdot h =&  P_{\bar g} ( - \Delta_{\bar g} + 1)^{-1}\left( D^2 \mathscr{G}_{M,\bar g}\cdot (h, \cdot) \right)
		\end{align*} and $DY_{\bar g}$ is an isomorphism on $\mathcal{E}_{\bar g}$ due to the fact that $D^2 \mathscr{G}_{M,\bar g}^{\mathcal S}$ is strictly negative definite on $\mathcal{E}_{\bar g}$ from Proposition \ref{prop:Einstein_second_var_non_positive_def}. 
		
		Since the functional $\mathscr{G}_{M,\bar g}^{\mathcal S}$ is dilation-invariant, applying Lemma \ref{lem:Morse_lemma}, we can find a neighborhood $U_{\bar g} \subseteq \mathcal{S}_{\bar g}$ of $\bar g$ such that for any $\hat g_{_S} \in U_{\bar g}$ satisfies 
		$$\mathscr{G}_{M,\bar g}^{\mathcal S}[\hat g_{_S}] \geq \mathscr{G}_{M,\bar g}^{\mathcal S} [\bar g],$$
		it implies $\hat g_{_S} \in \mathcal Q_{\bar g}$. That is, $\hat g_{_S} = c^2 \bar g$ for some constant $c > 0$.			
\end{proof}
\vskip 0.2in

Now we can prove the volume comparison of Einstein manifolds with respect to scalar curvature:
\begin{proof}[Proof of Theorem B]	
	According to Theorem \ref{thm:slice_thm}, we can find a local slice $\mathcal{S}_{\bar g}$ through the reference metric $\bar g$. Moreover, there exists a constant $\varepsilon_0 >0$ such that for any metric $\tilde g$ with $$||\tilde g - \bar g||_{C^2(M, \bar g)} < \varepsilon_0,$$
	we can find a diffeomorphism $\psi \in \mathscr{D}(M)$ with the property that $\psi^* \tilde g \in U_{\bar g} \subseteq \mathcal{S}_{\bar g}$, where the subset $U_{\bar g}$ is given by Proposition \ref{prop:Einstein_rigidity}.
	
	For $\lambda \neq 0$, suppose $g$ is a metric on $M$ with scalar curvature
	 $$R_g \geq n(n-1)\lambda$$
	and
	 $$||g - \bar g||_{C^2(M, \bar g)} < \varepsilon_0.$$ In addition, we assume the reversed inequality of the claimed volume comparison holds:
	\begin{align}\label{ineq:Einstein_reverse_vol_comparison}
		\lambda \left( V_M(g) - V_M(\bar g) \right) \geq 0.
	\end{align}
	This implies there is a diffeomorphism $\varphi \in \mathscr{D}(M)$ such that $\varphi^*g \in U_{\bar g} \subseteq \mathcal{S}_{\bar g}$ and 
	\begin{align*}
		\mathscr{G}_{M,\bar g}^{\mathcal{S}}[\varphi^*g] = V_M(\varphi^*g)^{\frac{2}{n}} \int_M (R_g \circ \varphi) dv_{\bar g} \geq V_M(\bar g)^{\frac{2}{n}} \int_M R_{\bar g} dv_{\bar g} = \mathscr{G}_{M,\bar g}^{\mathcal{S}} [\bar g]
	\end{align*}
	due to our assumptions and the fact that $R_{\bar g} = n(n-1) \lambda$ is a constant.
	According to Proposition \ref{prop:Einstein_rigidity}, there exists a constant $c >0$ such that $\varphi^*g = c^2 \bar g$. 
	
	From our assumptions, 
	\begin{align*}
		R_{\varphi^* g} = c^{-2} R_{\bar g} \geq R_{\bar g} = n(n-1)\lambda
	\end{align*}
	and hence
	\begin{align*}
		\lambda (1- c) \geq 0.
	\end{align*}
	However, inequality (\ref{ineq:Einstein_reverse_vol_comparison}) suggests that
	\begin{align*}
		 0 \leq \lambda \left( V_M(\varphi^*g) - V_M(\bar g) \right) = \lambda (c^n - 1) V_M(\bar g),
	\end{align*}
	which implies that $\lambda (1 - c) \leq 0$. Therefore, we conclude $c =1$ and hence $\varphi^*g = \bar g$. That is, $(M^n,g)$ is isometric to $(M^n, \bar g)$ and this concludes the theorem. 	
\end{proof}
\ \\

With analogous techniques, we can prove the local rigidity of Ricci-flat manifolds:
\begin{proof}[Proof of Theorem C]
	Similar to the proof of Theorem B, we can find a constant $\varepsilon_0 > 0$ such that for any metric $\tilde g$ satisfies
	\begin{align*}
		||\tilde g - \bar g||_{C^2(M, \bar g)} < \varepsilon_0,
	\end{align*}
	there exists a diffeomorphism $\varphi \in \mathscr{D}(M)$ such that $\varphi^*g \in U_{\bar g} \subseteq \mathcal{S}_{\bar g}$, where $U_{\bar g}$ is given in Proposition \ref{prop:Einstein_rigidity}.
	
	Suppose $g$ is a Riemannian metric with scalar curvature
	\begin{align*}
		R_g \geq 0
	\end{align*}
	and
	\begin{align*}
	||g - \bar g||_{C^2(M, \bar g)} < \varepsilon_0,
	\end{align*}
	then there is a diffeomorphism $\varphi \in \mathscr{D}(M)$ such that
	\begin{align*}
		\mathscr{G}_{M,\bar g}^{\mathcal{S}} [\varphi^* g] = V_M(\varphi^* g)^{\frac{2}{n}} \int_M (R_g \circ \varphi) dv_{\bar g} \geq 0.
	\end{align*}
	However, 
	\begin{align*}
	\mathscr{G}_{M,\bar g}^{\mathcal{S}} [\bar g] = V_M(\bar g)^{\frac{2}{n}} \int_M R_{\bar g} dv_{\bar g} = 0
	\end{align*}
	and hence there is a constant $c > 0$ such that $\varphi^*g = c^2 \bar g$ due to $\bar g$ is strictly stable Ricci-flat and Proposition \ref{prop:Einstein_rigidity}. Now the conclusion follows.
\end{proof}
\vskip 0.2in

According to Proposition \ref{prop:G_second_variations}, the second variation of $\mathscr{G}_{M,\bar g}$ at an unstable Einstein metric $\bar g$ is indefinite and hence $\bar g$ is a saddle point instead of a local maximum. It suggests that the volume comparison may fail for unstable Einstein manifolds and counter-examples can be constructed. It is well-known that a product of positive Einstein manifolds with identical Einstein constants is still Einstein but unstable (cf. \cite{Kroncke}). Due to this reason and its simple structure, it can be our first choice. \\

The following example is constructed by Macbeth (\cite{Macbeth}), which shows the stability assumption is necessary for our volume comparison theorem:
\begin{proposition}\label{prop:heather_example}
	There is a family of metrics $\{g_t\}_{t \in [0,1)}$ on $\mathbb{S}^2\times \mathbb{S}^2$ such that 
	\begin{itemize}
		\item $g_0$ is the canonical product metric on $\mathbb{S}^2\times \mathbb{S}^2$,
		\item $R_{g_t} = R_{g_{_{\mathbb{S}^2\times \mathbb{S}^2}}} = 4$, for all $t \in [0,1)$,
		\item $V_M(g_t) > V_M(g_{_{\mathbb{S}^2\times \mathbb{S}^2}})$ for all $t\in (0, 1)$.
	\end{itemize}
\end{proposition}
\begin{proof}
	Let $$g_t = (1+t)^{-1} g^1_{_{\mathbb{S}^2}} + (1-t)^{-1} g^2_{_{\mathbb{S}^2}}$$ with $t \in [0,1)$, where $g^i_{_{\mathbb{S}^2}}$ is the canonical metric on the $i^{th}$-$\mathbb{S}^2$ factor, $i=1,2$. It is easy to see that their scalar curvature is given by
	$$R_{g_t} = 2 (1+t) + 2 (1-t) = 4$$ for all $t \in [0,1)$. However, its volume is 
	$$V_{\mathbb{S}^2 \times \mathbb{S}^2} (g_t) = (1-t^2)^{-1} V_{\mathbb{S}^2 \times \mathbb{S}^2}(\bar g) > V_{\mathbb{S}^2 \times \mathbb{S}^2}(\bar g).$$
\end{proof}

It is straightforward to generalize this example to more general product cases. It would be interesting to see whether we can find an explicit example of unstable Einstein manifold, which is not of this type but the volume comparison fails. \\

\vskip 0.5in

\appendix

\section{Equivalence of Schoen's conjectures}

In this appendix, we show that two well-known conjectures proposed by Schoen \cite{Schoen} on hyperbolic manifolds actually are equivalent to each other. We believe the proof is known to experts. Unfortunately, we could not find an appropriate reference. Thus we present a proof here for readers who are interested in it.\\
  
  We start with a well-known concept in conformal geometry (cf. \cite{Viaclovsky}):
\begin{definition}
	For $n \geq 3$, let $(M^n, g)$ be a connected closed $n$-dimensional Riemannian manifold. The \emph{Yamabe constant} of the conformal class $[g]$ is defined to be 
	$$Y(M^n, [g]) := \inf_{g \in[g]} \frac{\int_M R_g dv_g}{(V_M(g))^{\frac{n-2}{n}}}. $$
	Moreover, we can define a min-max invariant
	$$Y(M^n) := \sup_{[g]}Y(M^n, [g])$$
	called \emph{Yamabe invariant} or \emph{$\sigma$-invariant}.
\end{definition}
\vskip 0.2in

It is well-known that  
$$Y(M^n) \leq Y(\mathbb{S}^n)$$
for any closed smooth manifold $M^n$ and the canonical spherical metric achieves the Yamabe invariant of $\mathbb{S}^n$. For a given closed hyperbolic manifold with dimension at least three, its hyperbolic metric is unique up to a dilation due to the well-known \emph{Mostow rigidity theorem} (cf. Theorem C.0 in \cite{Benedetti-Petronio}). Similar to the spherical case, Schoen conjectures that its Yamabe invariant is achieved by the canonical hyperbolic metric \cite{Schoen}:
\newtheorem*{SHYC}{\bf Conjecture A}
\begin{SHYC}[Schoen's hyperbolic Yamabe invariant conjecture]
	For $n\geq 3$, suppose $(M^n, \bar g)$ is an $n$-dimensional closed hyperbolic manifold, then
	$$Y(M^n) = Y(M^n, [\bar g]).$$
i.e. The Yamabe invariant is achieved by its canonical hyperbolic metric.
\end{SHYC}
\vskip 0.2in

Another conjecture about closed hyperbolic manifolds is the following one concerning volume comparison, which is also referred to be \emph{Schoen's conjecture}:
\newtheorem*{SHVC}{\bf Conjecture B}
\begin{SHVC}[Schoen's hyperbolic volume comparison conjecture]
	For $n\geq 3$, suppose $(M^n, \bar g)$ is an $n$-dimensional closed hyperbolic manifold, then for any metric $g$ on $M$ with scalar curvature $$R_g \geq R_{\bar g},$$ its volume satisfies that $$V_M(g) \geq V_M(\bar g).$$
\end{SHVC}
\vskip 0.2in

Obviously, Conjecture A involves all metrics on the given hyperbolic manifold and in general it is difficult to solve. Conjecture B only involves the comparison of a special metric with the reference metric, which seems easier to solve than Conjecture A.
However, Conjecture A and B are in fact equivalent to each other and hence they are equally difficult in this sense. The bright side of this equivalence is that we only need to solve Conjecture B then Conjecture A would hold automatically. This seems to be a promising approach to Conjecture A.\\  

In the rest of the appendix, we will show the equivalence of Conjecture A and B.\\  

We first show Conjecture A implies B. In order to do this, we need the following lemma adapted from an observation of Kobayashi \cite{Kobayashi}:
\begin{lemma}\label{lem:Yamabe_constant_comparison}
	Let $(M^n, g)$ be a closed manifold and $Y(M^n, [g])$ be the Yamabe constant of conformal class $[g]$. Then 
	$$- \left(\int_M |R_g^-|^{\frac{n}{2}} dv_g\right)^{\frac{2}{n}} \leq Y(M^n, [g]) \leq \left(\int_M |R_g^+|^{\frac{n}{2}} dv_g\right)^{\frac{2}{n}},$$
	where $R_g^+ := \max\{ R_g, 0\}$ and $R_g^- := \max\{ -R_g, 0\}$.
\end{lemma}

\begin{proof}
	By the conformal transformation law of scalar curvature
	$$Y(M^n, [g]) = \inf_{u > 0} \frac{ \int_M \left( a|\nabla_g u|_g^2 + R_g u^2 \right)dv_g}{\left(\int_M u^{\frac{2n}{n-2}} dv_g\right)^{\frac{n-2}{n}}},$$
	where $a := \frac{4(n-1)}{n-2}$. Then we have
	\begin{align*}
	Y(M^n, [g])
	\geq \inf_{u > 0} \frac{ \int_M  R_g u^2 dv_g}{\left(\int_M u^{\frac{2n}{n-2}} dv_g\right)^{\frac{n-2}{n}}} \geq - \inf_{u > 0} \frac{ \int_M R_g^- u^2 dv_g}{\left(\int_M u^{\frac{2n}{n-2}} dv_g\right)^{\frac{n-2}{n}}},
	\end{align*}
	since $R_g = R_g^+ - R_g^-$.
	By \emph{H\"older's inequality},
	$$\int_M  R_g^-  u^2 dv_g \leq \left( \int_M |R_g^-|^{\frac{n}{2}} dv_g\right)^{\frac{2}{n}}\left(\int_M u^{\frac{2n}{n-2}} dv_g\right)^{\frac{n-2}{n}}$$
	and hence
	$$Y(M^n, [g]) \geq - \left( \int_M |R_g^-|^{\frac{n}{2}} dv_g\right)^{\frac{2}{n}} . $$
	Similarly,
	$$Y(M^n, [g]) \leq \frac{ \int_M R_g dv_g}{\left(V_M(g)\right)^{\frac{n-2}{n}}} \leq \frac{ \int_M R_g^+ dv_g}{\left(V_M(g)\right)^{\frac{n-2}{n}}}.$$
	By \emph{H\"older's inequality},
	$$\int_M  R_g^+ dv_g \leq \left( \int_M |R_g^+|^{\frac{n}{2}} dv_g\right)^{\frac{2}{n}}\left(V_M(g)\right)^{\frac{n-2}{n}}$$
	and hence
	$$Y(M^n, [g]) \leq \left( \int_M |R_g^+|^{\frac{n}{2}} dv_g\right)^{\frac{2}{n}}.$$
\end{proof}

\vskip 0.2in

Immediately, it implies the following conformal volume comparison:
\begin{proposition}\label{prop:conformal_volume_comparison}
	Suppose $(M^n, \hat g)$ is a closed Riemannian manifold with strictly negative constant scalar curvature $R_{\hat g}$. Then for any metric $g\in [\hat g]$ with scalar curvature
	$$R_g \geq R_{\hat g},$$
	we have
	$$V_M(g) \geq V_M(\hat g).$$
\end{proposition}

\begin{proof}
	Since $R_{\hat g}$ is a strictly negative constant, then its Yamabe constant satisfies that $$Y(M^n, [\hat g]) < 0$$ and hence $\hat g$ is a Yamabe metric in the conformal class $[\hat g]$ due to the uniqueness of Yamabe metric of negative Yamabe constant. Thus,
	$$Y(M^n,[\hat g]) = R_{\hat g} (V_M(\hat g))^{\frac{2}{n}}.$$
	By Lemma \ref{lem:Yamabe_constant_comparison},
	$$(\min_M R_g) (V_M (g))^{\frac{2}{n}} \leq - \left(\int_M |R_g^-|^{\frac{n}{2}} dv_g\right)^{\frac{n}{2}} \leq Y(M^n, [\hat g]) = R_{\hat g} (V_M(\hat g))^{\frac{2}{n}}.$$
	Therefore,
	$$R_{\hat g} (V_M (g))^{\frac{2}{n}} \leq (\min_M R_g) (V_M (g))^{\frac{2}{n}} \leq R_{\hat g} (V_M(\hat g))^{\frac{2}{n}} $$
	and hence
	$$V_M(g) \geq V_M(\hat g).$$
\end{proof}

\vskip 0.2in

Now we are ready to show
\begin{proposition}
	$$\text{Conjecture A} \Rightarrow \text{Conjecture B}.$$ 
\end{proposition} 

\begin{proof}
	Let $(M^n, \bar g)$ be a closed hyperbolic manifold. Suppose $g$ is a metric on $M$ with scalar curvature 
	$$R_g \geq R_{\bar g}.$$
	We are going to show 
	$$V_M(g) \geq V_M(\bar g),$$
	assuming $\bar g$ achieves its Yamabe invariant $Y(M^n)$.
	
	From Conjecture A, the Yamabe constant of conformal class $[g]$ satisfies 
	$$Y(M^n, [g]) \leq Y(M^n) = Y(M^n, [\bar g]) < 0.$$
	Let $\hat g \in [g]$ be the unique Yamabe metric in $[g]$, which is normalized such that $R_{\hat g} = R_{\bar g}$. By Proposition \ref{prop:conformal_volume_comparison}, we have
	$$V_M(g) \geq V_M(\hat g).$$
	On the other hand, 
	$$R_{\hat g}V_M(\hat g) ^{\frac{2}{n}}= Y(M^n, [g]) \leq Y(M^n) = Y(M^n,[\bar g]) = R_{\bar g}V_M(\bar g)^{\frac{2}{n}},$$
	which implies
	$$V_M(\hat g) \geq V_M(\bar g).$$
	Therefore,
	$$V_M(g) \geq V_M(\hat g) \geq V_M(\bar g)$$
	and hence Conjecture B holds.
\end{proof}

\vskip 0.2in

Next we show the reverse implication is also true:
\begin{proposition}
	$$\text{Conjecture B} \Rightarrow \text{Conjecture A}.$$
\end{proposition}

\begin{proof}
	Let $(M^n, \bar g)$ be a closed hyperbolic manifold. We are going to show that its Yamabe invariant satisfies
	$$Y(M^n) = Y(M^n,[\bar g]),$$
	assuming the volume comparison holds.
	
	We first recall a classic result of Gromov and Lawson which states that there is no metric with non-negative scalar curvature on a compact hyperbolic manifold (see Corollary A in \cite{G-L_2}). That means, the Yamabe invariant
	$$Y(M^n) \leq 0$$
	and there is no metric on $M$ with identically vanishing scalar curvature. Thus for any metric $g$ on $M$, the Yamabe constant of conformal class $[g]$ is strictly negative:
	$$Y(M^n, [g]) < 0.$$
	
	Let $\hat g$ be the Yamabe metric in the conformal class $[g]$ with $R_{\hat g} = R_{\bar g} < 0$. According to Conjecture B, we have
	$$V_M(\hat g) \geq V_M(\bar g).$$
	Therefore, the Yamabe constant of $[g]$ satisfies that
	$$Y(M^n, [g]) = \frac{\int_M R_{\hat g} dv_{\hat g}}{(V_M(\hat g))^{\frac{n-2}{2}}} = R_{\hat g} (V_M(\hat g))^{\frac{2}{n}} \leq R_{\bar g} (V_M(\bar g))^{\frac{2}{n}} = Y(M^n,[\bar g]) .$$
	Since $g$ is arbitrary, we conclude
	$$Y(M^n) = \sup_{[g]} Y(M^n, [g]) = Y(M^n,[\bar g])$$
	and hence Conjecture A holds.
\end{proof}
\vskip 0.2in

In summary, we have the equivalence of Schoen's conjectures A and B:
\begin{theorem}\label{prop:SHVC_equiv_SHYC}
	$$\text{Conjecture A} \Leftrightarrow \text{Conjecture B}.$$
\end{theorem}
\ \\

\bibliographystyle{amsplain}

\end{document}